\documentclass[final,3p,times]{elsarticle}

\usepackage{graphicx}%
\usepackage{multirow}%
\usepackage{amsmath,amssymb,amsfonts}%
\usepackage{amsthm}%
\usepackage{mathrsfs}%
\usepackage[title]{appendix}%
\usepackage{xcolor}%
\usepackage{textcomp}%
\usepackage{manyfoot}%
\usepackage{booktabs}%
\usepackage{algorithm}%
\usepackage{algorithmicx}%
\usepackage{algpseudocode}%
\usepackage{listings}%
\usepackage{xcolor}
\usepackage{subcaption}

\newcommand{\corr}[1]{\textcolor{black}{#1}}

\usepackage{tikz}
\usepackage{pgfplots}
\pgfplotsset{compat=newest}

\theoremstyle{definition}

\theoremstyle{plain}
\newtheorem{theorem}{Theorem}[section]
\newtheorem{lemma}{Lemma}[section]
\newtheorem{remark}{Remark}[section]

\newtheorem{corollary}[theorem]{Corollary}
\newtheorem{example}{Example}

\numberwithin{equation}{section}

\usepackage[english]{babel}

\journal{Applied Numerical Mathematics}
\renewcommand{\corr}[1]{{\color{black}#1}}
\begin{document}

\begin{frontmatter}

\title{An optimal order  fractional backward collocation method for adjoint Volterra integro-differential equations}

\author[f4]{Mahmoud A. Zaky$^{*,}$}
  \ead{ma.zaky@yahoo.com;mibrahimm@imamu.edu.sa}

\address[f4]{Department of Mathematics and Statistics, College of Science, Imam Mohammad Ibn Saud Islamic University (IMSIU), Riyadh, Saudi Arabia}

\cortext[mycorrespondingauthor]{Corresponding author: M.A. Zaky}

\begin{abstract}
This paper develops and analyzes an optimal-order fractional backward collocation method for adjoint Volterra integro-differential equations with weakly singular kernels. The backward Volterra structure, together with the weakly singular kernel, induces fractional-power singularities at the terminal endpoint, thereby reducing the classical regularity of the exact solution and causing order deterioration in standard polynomial collocation methods. We first establish a regularity result that characterizes the terminal singular behavior of the solution, showing that the solution is continuously differentiable, whereas its second derivative may exhibit a weak singularity at the terminal endpoint. Motivated by this regularity structure, we introduce a terminally graded mesh and construct a fractional backward collocation scheme whose local approximation space is adapted to the endpoint singularity. Rigorous convergence and superconvergence estimates are derived for both the solution and its derivative. With suitable choices of the fractional parameter and the mesh-grading exponent, the proposed method attains the optimal convergence orders dictated by the local approximation degree. Numerical experiments confirm the theoretical predictions and demonstrate the accuracy and effectiveness of the method for adjoint weakly singular Volterra integro-differential equations.
\end{abstract}


\end{frontmatter}

\section{Introduction}

Weakly singular Volterra integro-differential equations have been extensively
studied in numerical analysis because Abel-type kernels often reduce the
regularity of the exact solution and lead to order reduction in standard
polynomial approximation methods; see, e.g.,
\cite{tang1992superconvergence,hu1996stieltjes,hu1998geometric,
brunner1983nonpolynomial,brunner1986polynomial}. Adjoint Volterra equations
form an important class within this framework. They are not merely formal dual
problems but arise naturally in error representation, duality arguments, and
a posteriori error estimation for Volterra equations. Shaw and
Whiteman~\cite{shaw1996discontinuous} showed that the adjoint Volterra
integral operator and the associated adjoint second-kind Volterra equation,
called the dual Volterra equation in their terminology, provide an important
tool for deriving a posteriori error estimates for discontinuous Galerkin
approximations of linear second-kind Volterra equations. This observation
places adjoint Volterra problems at the interface between the qualitative
theory of Volterra equations and the rigorous analysis of numerical methods.

In this paper, we study the adjoint weakly singular Volterra
integro-differential equation
\begin{equation}\label{eq:1.1R}
\begin{aligned}
\varphi'(x)
&=
a(x)\varphi(x)+\chi(x)
+
\int_x^b (s-x)^{\nu-1}\mathscr K(x,s)\varphi(s)\,ds,
\qquad 0\le x\le b,\\
\varphi(b)&=\varphi_b,
\end{aligned}
\end{equation}
where $0<\nu<1$. The kernel in \eqref{eq:1.1R} has an Abel-type singularity
along the diagonal $s=x$. Although this singularity is integrable, it
significantly affects the differentiability of the solution. Even when the data
$a$, $\chi$, and $\mathscr K$ are smooth, the solution of \eqref{eq:1.1R}
may exhibit limited regularity at the terminal endpoint. In particular,
terminal fractional powers are generated by the weakly singular Volterra
operator. This loss of regularity is the central analytical difficulty in the
numerical approximation of \eqref{eq:1.1R}.

Classical spectral and polynomial collocation methods are highly effective for
smooth problems, but their convergence behavior depends strongly on the
regularity of the exact solution. When endpoint singularities or limited
smoothness are present, the expected high-order or spectral accuracy may
deteriorate. Previous studies of weakly singular Volterra
integro-differential equations have shown that reduced regularity near a
singular endpoint can impose a significant order barrier on polynomial
collocation methods, particularly on uniform meshes
\cite{lubich1983runge,mustapha2013superconvergent,zhou2021block,
wei2012convergence}. For the adjoint problem considered here, the singularity
occurs at the terminal endpoint $b$, where higher derivatives of the solution
may become unbounded as $x\to b^{-}$.

Several endpoint-adapted strategies have been developed to address this
regularity-induced order reduction, including smoothing transformations and
nonpolynomial approximation spaces. Ameen et
al.~\cite{AmeenZakyDoha2021} developed a singularity-preserving spectral
collocation method for nonlinear adjoint Volterra integral equations arising
from terminal-value problems with right-sided Caputo fractional derivatives.
Their method combines a smoothing transformation with a Jacobi--Gauss
collocation scheme to accommodate the terminal singularity. Zaky
\cite{Zaky2026EndpointAdaptive} developed an endpoint-adaptive fractional
spectral collocation method for nonlinear adjoint Volterra integral equations
arising from tempered fractional terminal-value problems with nonsmooth
solutions. In that approach, backward fractional basis functions are adapted
directly to the terminal endpoint, avoiding an auxiliary smoothing
transformation and allowing both rational and irrational fractional orders.
Zaky et al.~\cite{ZakyAmeenAbuArqubDoha2025RightCaputo} proposed a high-order
fractional spectral collocation method for nonlinear adjoint Volterra integral
equations arising from right-sided Caputo terminal-value problems, using
backward fractional Legendre functions adapted to the terminal singularity.

More recently, Zaky~\cite{Zaky2026FractionalBackwardSpectral} introduced a
fractional backward spectral approximation framework for weakly singular
adjoint Volterra integral equations. That work constructs fractional backward
orthogonal functions whose approximation spaces reflect the terminal
singularity generated by weakly singular kernels and develops the associated
fractional backward Lagrange interpolation theory, error estimates, and
stability analysis. The present paper complements that
approximation-theoretic framework. Rather than treating adjoint integral
equations by a global spectral method, we develop an optimal-order fractional
backward collocation method for adjoint Volterra integro-differential
equations.

Much of the existing collocation analysis for weakly singular Volterra
problems focuses on singular behavior near the initial endpoint 
\cite{brunner2017volterra,brunner2004collocation,ma2024convergence}.
Here, both the approximation space and the graded mesh are instead designed
to resolve the weak singularity at the terminal endpoint $b$. The main
contributions of this paper are summarized as follows.
\begin{itemize}
\item We establish a terminal regularity result for the adjoint weakly singular
Volterra integro-differential equation \eqref{eq:1.1R}, showing that smooth
data may still generate terminal fractional-power components in the exact
solution.

\item We construct a fractional backward collocation framework adapted to this
terminal singular structure. A graded mesh refined toward $b$ is introduced,
and the local approximation space is generated by fractional powers in a
backward transformed variable.

\item We formulate the backward collocation scheme by approximating
$\psi=\varphi'$ and recovering $\varphi_h$ from the terminal condition by
backward integration. The resulting discretization is consistent with the
backward Volterra structure of the problem.

\item We derive global convergence estimates for $\psi$ and $\varphi$, together
with superconvergence estimates for $\varphi$, for $\psi$ at the collocation
points, and for $\psi$ at the regular endpoint. Suitable choices of the
fractional parameter $\mu$ and grading exponent $q$ allow the method to attain
the maximal orders predicted by the corresponding approximation and
superconvergence estimates.

\item Numerical experiments confirm the predicted convergence behavior and
illustrate the effectiveness of the method for adjoint weakly singular
Volterra integro-differential equations with terminal-endpoint singularities.
\end{itemize}

Although our analysis focuses on the prototypical adjoint weakly singular
Volterra integro-differential equation \eqref{eq:1.1R}, the proposed framework
also provides a basis for developing fractional backward collocation
methods for more general terminal-value and Volterra-type problems with weak
endpoint singularities.

The paper is organized as follows. Section~\ref{sec:solution-regularity}
establishes the terminal regularity structure of the solution of
\eqref{eq:1.1R}. Section~\ref{sec2} introduces the backward graded mesh,
fractional polynomial approximation spaces, and associated interpolation
operators. Section~\ref{sec3} presents the fractional backward collocation
method. Section~\ref{sec4} establishes the auxiliary estimates required for
the error analysis, and Section~\ref{sec5} derives the convergence and
superconvergence results. Section~\ref{sec:numericalR} presents numerical
experiments that verify the theoretical rates, and
Section~\ref{sec:conclusion} concludes the paper.

\medskip


\section{Regularity of the solution}
\label{sec:solution-regularity}

We first recall the regularity properties of the solution of the adjoint Volterra integro-differential equation \eqref{eq:1.1R}. In the present terminal formulation, the Volterra integral is taken over \([x,b]\); hence the weak singularity in the solution induced by the Abel-type kernel occurs at the terminal endpoint \(x=b\). The following theorem gives the corresponding differentiability estimate and terminal expansion.

\begin{theorem}
\label{thm:1.1R}
Let \(\Lambda:=[0,b]\), \(D_R:=\{(x,s):0\le x\le s\le b\}\), and \(\ell\in\mathbb N\). Assume that \(a,\chi\in C^\ell(\Lambda)\), \(\mathscr K\in C^\ell(D_R)\), and that \(\mathscr K(x,x)\ne0\) \corr{for all \(x\in\Lambda\).} If \(\varphi\) denotes the solution of the adjoint weakly singular Volterra integro-differential equation \eqref{eq:1.1R}, and if \(\psi:=\varphi'\), then
\begin{equation}\label{eq:u-regularity-mainR}
\varphi\in C^1(\Lambda)\cap C^{\ell+1}([0,b)).
\end{equation}
Moreover, there exists a constant \(C\), depending only on the data, such that
\begin{equation}\label{eq:u-derivative-boundR}
|\varphi^{(\ell+1)}(x)|\le C(b-x)^{\nu-\ell},\qquad 0\le x<b.
\end{equation}
The solution also admits the terminal expansion
\begin{equation}\label{eq:u-expansion-mainR}
\varphi(x)=\sum_{(i,k)\in\mathcal I_{\nu,\ell}^{R}}\alpha_{i,k}^{R}(\nu)(b-x)^{k(\nu+1)+i}+Y_{\ell+1}^{R}(x,\nu),\qquad x\in\Lambda,
\end{equation}
where $\mathcal I_{\nu,\ell}^{R}:=\{(i,k)\in\mathbb N_0^2:\ k(1+\nu)+i<1+\ell\}$, the coefficients \(\alpha_{i,k}^{R}(\nu)\) are determined by the data of the problem, and \(Y_{\ell+1}^{R}(\cdot,\nu)\in C^{\ell+1}(\Lambda)\). Consequently,
\begin{equation}\label{eq:v-regularity-mainR}
\psi\in C(\Lambda)\cap C^\ell([0,b)),\qquad |\psi^{(\ell)}(x)|\le C(b-x)^{\nu-\ell},\qquad 0\le x<b.
\end{equation}
\end{theorem}

\begin{proof}
Set \(\psi:=\varphi'\). From the terminal condition, we have $\varphi(x)=\varphi_b-\int_x^b\psi(s)\,ds$. Substitution into \eqref{eq:1.1R} gives
\begin{equation}\label{eq:v-equation-right-regularity}
\psi(x)=\chi_R(x)-\int_x^b\left[a(x)+\mathscr K_1^R(x,s)\right]\psi(s)\,ds,\qquad x\in\Lambda,
\end{equation}
where
\begin{equation}\label{eq:gR-regularity-def}
\chi_R(x):=\chi(x)+a(x)\varphi_b+\varphi_b\int_x^b(s-x)^{\nu-1}\mathscr K(x,s)\,ds,
\end{equation}
and
\begin{equation}\label{eq:K1R-def}
\mathscr K_1^R(x,s):=\int_x^s(\zeta-x)^{\nu-1}\mathscr K(x,\zeta)\,d\zeta,\qquad 0\le x\le s\le b.
\end{equation}

We next consider the \corr{behavior} of the free term \(\chi_R\) near the terminal point. Since \(\mathscr K\in C^\ell(D_R)\), its expansion near \((b,b)\) produces terms involving powers of \(b-x\) and \(b-s\). A typical contribution is
$$
\int_x^b(s-x)^{\nu-1}(b-x)^p(b-s)^q\,ds=(b-x)^{p+q+\nu}\int_0^1y^{\nu-1}(1-y)^q\,dy=B(\nu,q+1)(b-x)^{p+q+\nu}.
$$
Hence the weakly singular part of \(\chi_R\) generates terminal powers of the form $(b-x)^{i+\nu}$, \(i\in\mathbb N_0\). Consequently, $\chi_R\in C(\Lambda)\cap C^\ell([0,b))$ and $|\chi_R^{(\ell)}(x)|\le C(b-x)^{\nu-\ell}$, \(0\le x<b\).

The leading weak term is obtained by replacing \(\mathscr K(x,s)\) by \(\mathscr K(b,b)\) in the singular integral:
$$
\varphi_b\int_x^b(s-x)^{\nu-1}\mathscr K(x,s)\,ds=\frac{\corr{\varphi_b}\mathscr K(b,b)}{\nu}(b-x)^\nu+\text{higher-order terms}.
$$
Hence \(\psi=\varphi'\) generally contains a terminal term proportional to \((b-x)^\nu\), and therefore \(\psi'\) generally contains a term proportional to \((b-x)^{\nu-1}\).

Next we consider the kernel \(\mathscr K_1^R\). \corr{Using the boundedness of \(\mathscr K\) on \(D_R\),} we obtain
\begin{equation}\label{eq:K1R-bound}
|\mathscr K_1^R(x,s)|\le C\int_x^s(\zeta-x)^{\nu-1}\,d\zeta=\frac{C}{\nu}(s-x)^\nu,\qquad 0\le x\le s\le b.
\end{equation}
Hence \(\mathscr K_1^R\) is continuous on \(D_R\). Thus \((x,s)\mapsto a(x)+\mathscr K_1^R(x,s)\) is a Volterra kernel with the same terminal regularity scale as the free term \(\chi_R\). \corr{Applying the standard regularity argument to} \eqref{eq:v-equation-right-regularity} up to order \(\ell\), we obtain $\psi\in C(\Lambda)\cap C^\ell([0,b))$ and $|\psi^{(\ell)}(x)|\le C(b-x)^{\nu-\ell}$, \(0\le x<b\). This proves \eqref{eq:v-regularity-mainR}.

\corr{Since} \(\psi=\varphi'\) and \eqref{eq:v-regularity-mainR}, we obtain $\varphi\in C^1(\Lambda)\cap C^{\ell+1}([0,b))$ and, for \(0\le x<b\), $|\varphi^{(\ell+1)}(x)|=|\psi^{(\ell)}(x)|\le C(b-x)^{\nu-\ell}$. This proves \eqref{eq:u-regularity-mainR} and \eqref{eq:u-derivative-boundR}.

It remains to identify the exponents in the terminal expansion. The nonsingular terms in \eqref{eq:1.1R} act on terminal monomials according to
$$
\int_x^b t(\rho)(b-\rho)^\alpha\,d\rho=\sum_{p=0}^{M}\frac{t_p}{p+\alpha+1}(b-x)^{p+\alpha+1}+O\bigl((b-x)^{M+\alpha+2}\bigr),
$$
where $t(\rho)=\sum_{p=0}^{M}t_p(b-\rho)^p+O((b-\rho)^{M+1})$. For the weakly singular term, one obtains
$$
\int_x^b\int_\rho^b(s-\rho)^{\nu-1}(b-\rho)^p(b-s)^{q+\alpha}\,ds\,d\rho=\frac{B(\nu,q+\alpha+1)}{p+q+\alpha+\nu+1}(b-x)^{p+q+\alpha+\nu+1}.
$$
Thus the weakly singular integro-differential term, together with the recovery of \(\varphi\) from \(\psi=\varphi'\), shifts the terminal exponent by \(1+\nu\), up to additional integer powers. Hence the exponents occurring in the terminal expansion of \(\varphi\) are of the form $i+k(1+\nu)$, \(i,k\in\mathbb N_0\).

\corr{Iterating this argument,} keeping all powers below \(\ell+1\), and absorbing the remainder into \(C^{\ell+1}(\Lambda)\), we obtain
$$
\varphi(x)=\sum_{(i,k)\in\mathcal I_{\nu,\ell}^{R}}\alpha_{i,k}^{R}(\nu)(b-x)^{i+k(1+\nu)}+Y_{\ell+1}^{R}(x,\nu),
$$
where $\mathcal I_{\nu,\ell}^{R}=\{(i,k)\in\mathbb N_0^2:\ i+k(1+\nu)<\ell+1\}$ and $Y_{\ell+1}^{R}(\cdot,\nu)\in C^{\ell+1}(\Lambda)$. This proves \eqref{eq:u-expansion-mainR}.

Differentiating \eqref{eq:u-expansion-mainR} gives the corresponding terminal structure of \(\psi=\varphi'\). In particular, the term \((b-x)^{1+\nu}\) in \(\varphi\) generates the factor \((b-x)^\nu\) in \(\psi\), and hence the factor \((b-x)^{\nu-1}\) in \(\psi'\). Since \(0<\nu<1\), this shows that, in general, $\varphi\in C^1(\Lambda)$ and $\varphi\notin C^2(\Lambda)$.
\end{proof}

\section{Backward graded mesh and fractional interpolation}
\label{sec2}

Let \(\mathscr{M}\in\mathbb{N}\) and \(q\ge1\). We introduce the terminally graded mesh
\[
\mathscr{P}_{\mathscr{M}}
:=
\left\{
x_j
=
b\left(1-\left(\frac{\mathscr{M}-j}{\mathscr{M}}\right)^q\right),
\quad j=0,1,\ldots,\mathscr{M}
\right\}.
\]
For \(q=1\), the mesh is uniform; for \(q>1\), the mesh points are clustered
toward the terminal endpoint \(b\), where the singular \corr{behavior} described in
Theorem~\ref{thm:1.1R} may occur.

Let \(0<\mu<1\). For \(j=1,\ldots,\mathscr{M}\), set
\(\varrho_j=(x_{j-1},x_j]\) and \(h_j=x_j-x_{j-1}\), and define the transformed
mesh by \(\rho_j=b-(b-x_j)^\mu\),
\(\varrho_{j,\mu}=(\rho_{j-1},\rho_j]\), and
\(h_{j,\mu}=\rho_j-\rho_{j-1}\). We also write
\(h=\max_{1\le j\le\mathscr{M}}h_j\).

The local fractional polynomial space is defined by
\cite{Zaky2026FractionalBackwardSpectral}
\[
P_n^\mu
:=
\operatorname{span}\{1,(b-x)^\mu,\ldots,(b-x)^{n\mu}\},
\qquad n\in\mathbb{N}_0 .
\]
For a fixed positive integer \(\flat\), define
\[
S_\flat^\mu(\mathscr{P}_{\mathscr{M}})
:=
\{f(x):f|_{\varrho_j}\in P_{\flat-1}^\mu,\quad
j=1,\ldots,\mathscr{M}\},
\]
and
\[
\widehat S_\flat^\mu(\mathscr{P}_{\mathscr{M}})
:=
\{f(\rho):f|_{\varrho_{j,\mu}}\in P_{\flat-1}^1,\quad
j=1,\ldots,\mathscr{M}\}.
\]

Let \(0\le\xi_1<\cdots<\xi_\flat\le1\) be the collocation parameters. On
\(\varrho_{j,\mu}\), set
\(\eta_{j,l}=\rho_{j-1}+\xi_lh_{j,\mu}\),
\(j=1,\ldots,\mathscr{M}\), \(l=1,\ldots,\flat\).
Define \(\tau(\rho):=b-(b-\rho)^{1/\mu}\). The corresponding collocation
points in the original variable are
\begin{equation}\label{eq:2.1R}
X_j
:=
\{x_{j,l}:x_{j,l}=\tau(\eta_{j,l}),\ l=1,\ldots,\flat\},
\qquad j=1,\ldots,\mathscr{M}.
\end{equation}

The fractional backward Lagrange basis on \(\varrho_j\) is given by
\cite{Zaky2026FractionalBackwardSpectral}:
\[
L_{j,l}^{\mu}(x)
:=
\prod_{\substack{i=1\\ i\ne l}}^{\flat}
\frac{(b-x)^{\mu}-(b-x_{j,i})^{\mu}}
     {(b-x_{j,l})^{\mu}-(b-x_{j,i})^{\mu}},
\qquad x\in\varrho_j,
\]
for \(j=1,\ldots,\mathscr{M}\) and \(l=1,\ldots,\flat\). Hence, for
\(f\in C(\Lambda)\), the interpolation operator
\(I_{\flat,\mathscr{M}}^\mu:C(\Lambda)\to
S_\flat^\mu(\mathscr{P}_{\mathscr{M}})\) is defined locally by
\[
\left(I_{\flat,\mathscr{M}}^\mu f\right)(x)
:=
\sum_{l=1}^{\flat}L_{j,l}^\mu(x)f(x_{j,l}),
\qquad x\in\varrho_j .
\]

The transformation \(x=\tau(\rho)\) reduces the fractional interpolation to
ordinary polynomial interpolation. Indeed, for \(x\in\varrho_j\),
\begin{equation}\label{eq:2.2R}
L_{j,l}^\mu(\tau(\rho))
=
\prod_{\substack{i=1\\ i\ne l}}^{\flat}
\frac{\rho-\eta_{j,i}}{\eta_{j,l}-\eta_{j,i}}
=:L_{j,l}(\rho),
\qquad \rho\in\varrho_{j,\mu}.
\end{equation}
Here \(L_{j,l}(\rho)\) is the standard Lagrange basis associated with
\(\{\eta_{j,l}\}_{l=1}^{\flat}\) on \(\varrho_{j,\mu}\).

For \(f\in C[b-b^\mu,b]\), define the interpolation operator
\(Q_{\flat,\mathscr{M}}^\mu:C[b-b^\mu,b]\to
\widehat S_\flat^\mu(\mathscr{P}_{\mathscr{M}})\) by
\[
\left(Q_{\flat,\mathscr{M}}^\mu f\right)(\rho)
:=
\sum_{l=1}^{\flat}L_{j,l}(\rho)f(\eta_{j,l}),
\qquad \rho\in\varrho_{j,\mu}.
\]

It follows from \eqref{eq:2.2R} that, when \(x=\tau(\rho)\),
\begin{align}
\left(I_{\flat,\mathscr{M}}^\mu\psi\right)(x)\big|_{x\in\varrho_j}
&=
\sum_{l=1}^{\flat}L_{j,l}^\mu(x)\psi(x_{j,l})
=
\sum_{l=1}^{\flat}L_{j,l}(\rho)\psi(\tau(\eta_{j,l}))
\nonumber\\
&=
\left(Q_{\flat,\mathscr{M}}^\mu\widehat{\psi}\right)(\rho)
\big|_{\rho\in\varrho_{j,\mu}},
\label{eq:2.3R}
\end{align}
where \(\widehat{\psi}(\rho):=\psi(\tau(\rho))\).

We next present two mesh estimates that will be used repeatedly in the error
analysis.

\begin{lemma}\label{lem:2.1R}
Let \(\flat\in\mathbb{N}\) \corr{and \(\mathscr{M}\ge2\)}. Then there exists
a constant \(C\), independent of \(\mathscr{M}\) and \(j\), such that, for
\(1\le j\le\mathscr{M}\),
\[
h_{j,\mu}^{\flat}(b-x_{j-1})^{\mu\varsigma}
\le
C\mathscr{M}^{-q\mu(\flat+\varsigma)}
(\mathscr{M}-j+1)^{q\mu(\flat+\varsigma)-\flat}
\le
C\mathscr{M}^{-\min\{q\mu(\flat+\varsigma),\flat\}},
\qquad \varsigma\in\mathbb{R},
\]
and
\[
h_{j,\mu}^{\flat}
\left|\ln(b-x_{j-1})^\mu\right|
\le
C\mathscr{M}^{-q\mu\flat}
(\mathscr{M}-j+1)^{q\mu\flat-\flat}\ln\mathscr{M}
\le
C\mathscr{M}^{-\min\{q\mu\flat,\flat\}}\ln\mathscr{M}.
\]
\end{lemma}

\begin{proof}
For \(j=\mathscr{M}\), the identities
\(h_{\mathscr{M},\mu}^{\flat}
=b^{\mu\flat}\mathscr{M}^{-q\mu\flat}\) and
\((b-x_{\mathscr{M}-1})^{\mu\varsigma}
=b^{\mu\varsigma}\mathscr{M}^{-q\mu\varsigma}\)
give the first bound immediately.

Let \(1\le j\le\mathscr{M}-1\). From the definitions of \(x_j\) and
\(\rho_j\),
\[
h_{j,\mu}^{\flat}
=
b^{\mu\flat}\mathscr{M}^{-q\mu\flat}
\left((\mathscr{M}-j+1)^{q\mu}
-(\mathscr{M}-j)^{q\mu}\right)^\flat,
\]
and
\[
(b-x_{j-1})^{\mu\varsigma}
=
b^{\mu\varsigma}\mathscr{M}^{-q\mu\varsigma}
(\mathscr{M}-j+1)^{q\mu\varsigma}.
\]
Hence
\[
h_{j,\mu}^{\flat}(b-x_{j-1})^{\mu\varsigma}
=
C\mathscr{M}^{-q\mu(\flat+\varsigma)}
(\mathscr{M}-j+1)^{q\mu\varsigma}
\left((\mathscr{M}-j+1)^{q\mu}
-(\mathscr{M}-j)^{q\mu}\right)^\flat.
\]
By the \corr{mean value theorem},
\[
(\mathscr{M}-j+1)^{q\mu}
-(\mathscr{M}-j)^{q\mu}
\le
C(\mathscr{M}-j+1)^{q\mu-1}.
\]
Therefore,
\[
h_{j,\mu}^{\flat}(b-x_{j-1})^{\mu\varsigma}
\le
C\mathscr{M}^{-q\mu(\flat+\varsigma)}
(\mathscr{M}-j+1)^{q\mu(\flat+\varsigma)-\flat}.
\]
Taking the maximum over \(1\le\mathscr{M}-j+1\le\mathscr{M}\) gives
\[
h_{j,\mu}^{\flat}(b-x_{j-1})^{\mu\varsigma}
\le
C\mathscr{M}^{-\min\{q\mu(\flat+\varsigma),\flat\}}.
\]

For the logarithmic estimate, note that
\[
\left|\ln(b-x_{j-1})^\mu\right|
=
\left|
\ln\left(
b^\mu
\left(\frac{\mathscr{M}-j+1}{\mathscr{M}}\right)^{q\mu}
\right)
\right|
\le C\ln\mathscr{M}.
\]
Combining this with the first estimate for \(\varsigma=0\) yields
\[
h_{j,\mu}^{\flat}
\left|\ln(b-x_{j-1})^\mu\right|
\le
C\mathscr{M}^{-q\mu\flat}
(\mathscr{M}-j+1)^{q\mu\flat-\flat}\ln\mathscr{M}
\le
C\mathscr{M}^{-\min\{q\mu\flat,\flat\}}\ln\mathscr{M}.
\]
\end{proof}

We now introduce the regularity class used for functions with terminal
fractional \corr{behavior}; see
\cite{Zaky2026FractionalBackwardSpectral}. Let \(\flat\in\mathbb{N}\) and
\(\gamma>0\). The space
\(C^{\flat,1-\gamma}[b-b^\mu,b)\) consists of all real-valued functions
\(f\in C[b-b^\mu,b]\cap C^\flat[b-b^\mu,b)\) such that,
\corr{for some constants \(c_k=c_k(f)\),}
\[
|f^{(k)}(\rho)|
\le
c_k
\begin{cases}
1, & k<\gamma,\\
1+|\ln(b-\rho)|, & k=\gamma,\\
(b-\rho)^{\gamma-k}, & k>\gamma,
\end{cases}
\qquad b-b^\mu\le\rho<b,\quad k=0,\ldots,\flat.
\]
For instance, if \(f(\rho)=(b-\rho)^\gamma\widetilde f(\rho)\), with
\(\widetilde f\in C^\flat[b-b^\mu,b]\), then
\(f\in C^{\flat,1-\gamma}[b-b^\mu,b)\).

The following lemma gives the local interpolation error for
\(Q_{\flat,\mathscr{M}}^\mu\) on the transformed mesh.

\begin{lemma}\label{lem:2.2R}
For each \(f\in C^{\flat,1-\gamma}[b-b^\mu,b)\), there exists a constant
\(C=C(f)\), independent of \(\mathscr{M}\) and \(j\), such that, for
\(j=1,\ldots,\mathscr{M}\),
\[
\max_{\rho_{j-1}\le\rho\le\rho_j}
|f(\rho)-Q_{\flat,\mathscr{M}}^\mu f(\rho)|
\le
Ch_{j,\mu}^{\flat}
\begin{cases}
1, & \flat<\gamma,\\[1mm]
1+\left|\ln(b-\rho_{j-1})\right|, & \flat=\gamma,\\[1mm]
(b-\rho_{j-1})^{\gamma-\flat}, & \flat>\gamma.
\end{cases}
\]
Consequently,
\[
\max_{\rho_{j-1}\le\rho\le\rho_j}
|f(\rho)-Q_{\flat,\mathscr{M}}^\mu f(\rho)|
\le
C
\begin{cases}
\mathscr{M}^{-\min\{q\mu\flat,\flat\}}, & \flat<\gamma,\\[1mm]
\mathscr{M}^{-\min\{q\mu\flat,\flat\}}\ln\mathscr{M},
& \flat=\gamma,\\[1mm]
\mathscr{M}^{-\min\{q\mu\gamma,\flat\}}, & \flat>\gamma.
\end{cases}
\]
In particular, for \(j=1\),
\begin{equation}\label{eq:2.4R}
\max_{\rho_0\le\rho\le\rho_1}
|f(\rho)-Q_{\flat,\mathscr{M}}^\mu f(\rho)|
\le
C\mathscr{M}^{-\flat}.
\end{equation}
\end{lemma}

\begin{proof}
Fix \(j\). Since \(Q_{\flat,\mathscr{M}}^\mu\) is the interpolation operator
on \(\varrho_{j,\mu}\) and reproduces \(P_{\flat-1}^{1}\), for every
\(p\in P_{\flat-1}^{1}\),
\[
f-Q_{\flat,\mathscr{M}}^\mu f
=
(f-p)-Q_{\flat,\mathscr{M}}^\mu(f-p).
\]
\corr{Since the interpolation nodes are affine images of the fixed
parameters \(\{\xi_l\}_{l=1}^{\flat}\), the corresponding Lebesgue constant
is independent of \(j\) and \(\mathscr{M}\).} Thus,
\[
\max_{\rho\in\varrho_{j,\mu}}
|f(\rho)-Q_{\flat,\mathscr{M}}^\mu f(\rho)|
\le
C\max_{\rho\in\varrho_{j,\mu}}|f(\rho)-p(\rho)|.
\]
Taking \(p\) as the Taylor polynomial of degree \(\flat-1\) at
\(\rho_{j-1}\), we obtain
\[
f(\rho)-p(\rho)
=
\frac{1}{(\flat-1)!}
\int_{\rho_{j-1}}^\rho
(\rho-s)^{\flat-1}f^{(\flat)}(s)\,ds.
\]
Using the definition of \(C^{\flat,1-\gamma}[b-b^\mu,b)\),
\corr{and estimating the terminal interval directly when
\(j=\mathscr{M}\) and \(\flat\ge\gamma\),} this gives
\[
|f(\rho)-p(\rho)|
\le
Ch_{j,\mu}^{\flat}
\begin{cases}
1, & \flat<\gamma,\\[1mm]
1+|\ln(b-\rho_{j-1})|, & \flat=\gamma,\\[1mm]
(b-\rho_{j-1})^{\gamma-\flat}, & \flat>\gamma.
\end{cases}
\]
This proves the local estimate.

Since \(b-\rho_{j-1}=(b-x_{j-1})^\mu\), the three
\(\mathscr{M}\)-dependent estimates follow from Lemma~\ref{lem:2.1R},
with \(\varsigma=0\) in the first case, with the logarithmic estimate in
the second case, and with \(\varsigma=\gamma-\flat\) in the third case.
\corr{Since \(b-\rho_1\) is bounded away from zero and
\(h_{1,\mu}\le C\mathscr{M}^{-1}\),} estimate \eqref{eq:2.4R} follows by
taking \(j=1\).
\end{proof}

We shall also use the following weighted version of the preceding estimate.

\begin{lemma}\label{lem:2.3R}
For each \(f\in C^{\flat,1-\gamma}[b-b^\mu,b)\), there exists a constant
\(C=C(f)\), independent of \(\mathscr{M}\) and \(j\), such that, for
\(j=1,\ldots,\mathscr{M}\),
\[
h_j
\max_{\rho_{j-1}\le\rho\le\rho_j}
|f(\rho)-Q_{\flat,\mathscr{M}}^\mu f(\rho)|
\le
C
\begin{cases}
\mathscr{M}^{-\min\{q(\mu\flat+1),\flat+1\}}, & \flat<\gamma,\\[1mm]
\mathscr{M}^{-\min\{q(\mu\flat+1),\flat+1\}}\ln\mathscr{M},
& \flat=\gamma,\\[1mm]
\mathscr{M}^{-\min\{q(\mu\gamma+1),\flat+1\}}, & \flat>\gamma.
\end{cases}
\]
\end{lemma}

\begin{proof}
By the \corr{mean value theorem} applied to the graded mesh,
\[
h_j
=
b\mathscr{M}^{-q}
\left((1+\mathscr{M}-j)^q-(\mathscr{M}-j)^q\right)
\le
C\mathscr{M}^{-q}(1+\mathscr{M}-j)^{q-1}.
\]
Combining this with Lemmas~\ref{lem:2.2R} and~\ref{lem:2.1R} gives
\[
h_j
\max_{\rho_{j-1}\le\rho\le\rho_j}
|f(\rho)-Q_{\flat,\mathscr{M}}^\mu f(\rho)|
\le
C
\begin{cases}
\mathscr{M}^{-q(\mu\flat+1)}
(\mathscr{M}-j+1)^{q(\mu\flat+1)-\flat-1},
& \flat<\gamma,\\[1mm]
\mathscr{M}^{-q(\mu\flat+1)}
(\mathscr{M}-j+1)^{q(\mu\flat+1)-\flat-1}\ln\mathscr{M},
& \flat=\gamma,\\[1mm]
\mathscr{M}^{-q(\mu\gamma+1)}
(\mathscr{M}-j+1)^{q(\mu\gamma+1)-\flat-1},
& \flat>\gamma.
\end{cases}
\]
Taking the maximum over \(1\le1+\mathscr{M}-j\le\mathscr{M}\) yields the
asserted bounds.
\end{proof}

\section{Fractional backward collocation method}
\label{sec3}

We now construct the fractional backward collocation method for
\eqref{eq:1.1R}. The method approximates \(\psi=\varphi'\) in
\(S_\flat^\mu(\mathscr{P}_{\mathscr{M}})\). The approximation \(\varphi_h\) is then recovered
from the terminal value, so that the computation proceeds from \(x_{\mathscr{M}}=b\)
backward to \(x_0=0\). \corr{We set \(\varphi_h(x_{\mathscr{M}})=\varphi_b\).}

On each interval \(\varrho_j\), write
\begin{equation}\label{eq:3.1aR}
\psi_h(x)
=
\sum_{k=1}^{\flat}Y_{j,k}L_{j,k}^{\mu}(x),
\qquad
Y_{j,k}:=\psi_h(x_{j,k}).
\end{equation}
The corresponding approximation of \(\varphi\) is defined by
\begin{equation}\label{eq:3.1bR}
\varphi_h(x)
=
\varphi_h(x_j)-\int_x^{x_j}\psi_h(s)\,ds,
\qquad x\in[x_{j-1},x_j].
\end{equation}
Equivalently, with
\(t_{j,k}(x):=\int_x^{x_j}L_{j,k}^{\mu}(s)\,ds\), we have
\begin{equation}\label{eq:3.2R}
\varphi_h(x)
=
\varphi_h(x_j)-\sum_{k=1}^{\flat}t_{j,k}(x)Y_{j,k},
\qquad x\in\varrho_j.
\end{equation}

The collocation equations are imposed at the points \(x_{j,l}\). Thus, for
\(l=1,\ldots,\flat\),
\begin{align}
Y_{j,l}
&=
a(x_{j,l})\varphi_h(x_{j,l})+\chi(x_{j,l})
+
\sum_{i=j+1}^{\mathscr{M}}
\int_{x_{i-1}}^{x_i}
(s-x_{j,l})^{\nu-1}\mathscr K(x_{j,l},s)\varphi_h(s)\,ds
\nonumber\\
&\quad
+
\int_{x_{j,l}}^{x_j}
(s-x_{j,l})^{\nu-1}\mathscr K(x_{j,l},s)\varphi_h(s)\,ds .
\label{eq:3.3R-coll}
\end{align}
Substituting \eqref{eq:3.2R} into the current-interval contribution in
\eqref{eq:3.3R-coll} gives
\begin{align}
Y_{j,l}
&=
a(x_{j,l})\varphi_h(x_j)
-
a(x_{j,l})\sum_{k=1}^{\flat}t_{j,k}(x_{j,l})Y_{j,k}
+
\chi(x_{j,l})
+
F_j(x_{j,l})
\nonumber\\
&\quad
+
\varphi_h(x_j)
\int_{x_{j,l}}^{x_j}
(s-x_{j,l})^{\nu-1}\mathscr K(x_{j,l},s)\,ds
\nonumber\\
&\quad
-
\sum_{k=1}^{\flat}Y_{j,k}
\int_{x_{j,l}}^{x_j}
(s-x_{j,l})^{\nu-1}\mathscr K(x_{j,l},s)t_{j,k}(s)\,ds ,
\label{eq:3.3R-expanded}
\end{align}
where
\[
F_j(x_{j,l})
:=
\sum_{i=j+1}^{\mathscr{M}}
\int_{x_{i-1}}^{x_i}
(s-x_{j,l})^{\nu-1}\mathscr K(x_{j,l},s)\varphi_h(s)\,ds,
\]
\corr{with \(F_{\mathscr{M}}(x_{\mathscr{M},l})=0\).}

Define
\[
A_j:=\operatorname{diag}\bigl(a(x_{j,l})\bigr)_{l=1}^{\flat},
\qquad
B_j:=\bigl(t_{j,k}(x_{j,l})\bigr)_{l,k=1}^{\flat},
\]
\[
C_j:=
\left(
\int_{x_{j,l}}^{x_j}
(z-x_{j,l})^{\nu-1}\mathscr K(x_{j,l},z)t_{j,k}(z)\,dz
\right)_{l,k=1}^{\flat},
\]
and
\[
Y_j:=(Y_{j,1},\ldots,Y_{j,\flat})^{\corr{\mathrm T}},
\qquad
\boldsymbol{\chi}_j:=
(\chi(x_{j,1}),\ldots,\chi(x_{j,\flat}))^{\corr{\mathrm T}},
\]
\[
H_j:=
(F_j(x_{j,1}),\ldots,F_j(x_{j,\flat}))^{\corr{\mathrm T}},
\]
\[
\kappa_j:=
\left(
a(x_{j,l})
+
\int_{x_{j,l}}^{x_j}
(s-x_{j,l})^{\nu-1}\mathscr K(x_{j,l},s)\,ds
\right)_{l=1}^{\flat}.
\]
Then \eqref{eq:3.3R-expanded} is equivalent to
\begin{equation}\label{eq:3.3R}
\left(I_\flat+A_jB_j+C_j\right)Y_j
=
\boldsymbol{\chi}_j+H_j+\varphi_h(x_j)\kappa_j,
\qquad j=\mathscr{M},\mathscr{M}-1,\ldots,1,
\end{equation}
where \(I_\flat\) is the identity matrix in \(\mathbb R^\flat\). Once \(Y_j\) is
computed, the left endpoint value is updated by
\begin{equation}\label{eq:3.4R}
\varphi_h(x_{j-1})
=
\varphi_h(x_j)-\sum_{k=1}^{\flat}t_{j,k}(x_{j-1})Y_{j,k}.
\end{equation}

\section{Auxiliary lemmas}
\label{sec4}

\corr{In this section, we present} several auxiliary estimates needed for the error analysis. For \(r\in\mathbb{R}\), let \(\lfloor r\rfloor\) denote the greatest integer less than or equal to \(r\).

\begin{lemma}\label{lem:4.1R}
Let \(a\in(-1,0]\) and \(t\in\mathbb{R}\). For \(1\le j\le \mathscr{M}-2\),
there exists a constant \(C=C(a,t)\), independent of \(j\) and \(\mathscr{M}\), such
that
\[
\sum_{i=j+1}^{\mathscr{M}-1}(i-j)^a(1+\mathscr{M}-i)^t
\le
C
\begin{cases}
(\mathscr{M}-j)^a, & t<-1,\\[1mm]
(\mathscr{M}-j)^a\ln(\mathscr{M}-j+1), & t=-1,\\[1mm]
(\mathscr{M}-j)^{a+t+1}, & t>-1.
\end{cases}
\]
\end{lemma}

\begin{proof}
Set \(r=i-j\) and \(K=\mathscr{M}-j+1\). Then \(1\le r\le K-2\) and
\(1+\mathscr{M}-i=K-r\). Hence
\[
\sum_{i=j+1}^{\mathscr{M}-1}(i-j)^a(1+\mathscr{M}-i)^t
=
\sum_{r=1}^{K-2}r^a(K-r)^t
\le
\sum_{r=1}^{K-1}r^a(K-r)^t.
\]
Let \(\corr{K_0}:=\lfloor K/2\rfloor\) and write
\[
\sum_{r=1}^{K-1}r^a(K-r)^t=S_1+S_2,
\]
where
\[
S_1=\sum_{r=1}^{\corr{K_0}}r^a(K-r)^t,\qquad
S_2=\sum_{r=\corr{K_0}+1}^{K-1}r^a(K-r)^t.
\]
If \(1\le r\le \corr{K_0}\), then \(K/2\le K-r\le K\), and hence
\((K-r)^t\le CK^t\). Since \(a>-1\),
\[
S_1
\le
CK^t\sum_{r=1}^{\corr{K_0}}r^a
\le
CK^{a+t+1}.
\]
If \(\corr{K_0}+1\le r\le K-1\), then \(r>K/2\). Since \(a\le0\), this gives
\(r^a\le CK^a\). Therefore, with \(l=K-r\),
\[
S_2
\le
CK^a\sum_{r=\corr{K_0}+1}^{K-1}(K-r)^t
\le
CK^a\sum_{l=1}^{K-1}l^t.
\]
Using the standard bound
\[
\sum_{l=1}^{K-1}l^t
\le
C
\begin{cases}
1, & t<-1,\\
\ln K, & t=-1,\\
K^{t+1}, & t>-1,
\end{cases}
\]
we obtain
\[
S_2
\le
C
\begin{cases}
K^a, & t<-1,\\[1mm]
K^a\ln K, & t=-1,\\[1mm]
K^{a+t+1}, & t>-1.
\end{cases}
\]
Moreover, the estimate for \(S_1\) is absorbed by the same bounds. Indeed,
if \(t<-1\), then \(K^{a+t+1}\le K^a\); if \(t=-1\), then
\(K^{a+t+1}=K^a\le K^a\ln K\); and if \(t>-1\), the bound is already
\(K^{a+t+1}\). Thus
\[
\sum_{r=1}^{K-1}r^a(K-r)^t
\le
C
\begin{cases}
K^a, & t<-1,\\[1mm]
K^a\ln K, & t=-1,\\[1mm]
K^{a+t+1}, & t>-1.
\end{cases}
\]
Since \(K=\mathscr{M}-j+1\) and
\(\mathscr{M}-j\le K\le2(\mathscr{M}-j)\), the asserted estimate follows.
\end{proof}

For \(\corr{r>0}\), \(s>0\), and \(\mathscr{M}\in\mathbb{N}\), we define
\[
\varpi(\corr{r},s;\mathscr{M})
=
\begin{cases}
\mathscr{M}^{-\min\{\corr{r},s\}}, & \corr{r}\ne s,\\[1mm]
\mathscr{M}^{-\min\{\corr{r},s\}}\ln\mathscr{M}, & \corr{r}=s.
\end{cases}
\]
The following lemma is stated under the assumption that the collocation
\corr{parameters} \(\{\xi_l\}_{l=1}^\flat\) satisfy
\begin{equation}\label{eq:4.1R}
\mathcal{J}:=
\int_0^1\prod_{l=1}^\flat(\zeta-\xi_l)\,d\zeta=0.
\end{equation}

The following lemma gives an estimate for the accumulated interpolation error over the backward Volterra interval, which will be used in the superconvergence analysis.

\begin{lemma}\label{lem:4.2R}
Let \(\flat\in\mathbb{N}\) and \(\gamma>0\). Assume that the collocation
\corr{parameters} \(\{\xi_l\}_{l=1}^\flat\) satisfy \eqref{eq:4.1R}, and set
\(\lambda=1/\mu\). If
\(f\in C^{\flat+1,1-\gamma}[b-b^\mu,b)\), then, for
\(x\in\varrho_j\), \(j=1,\ldots,\mathscr{M}\), define
\[
\Delta_{j,i;x}
=
\begin{cases}
\displaystyle
\int_{\rho_{i-1}}^{\rho_i}
\bigl(f(\rho)-Q_{\flat,\mathscr{M}}^\mu f(\rho)\bigr)
(b-\rho)^{\lambda-1}\,d\rho,
& j<i\le \mathscr{M},\\[4mm]
\displaystyle
\int_{\rho(x)}^{\rho_j}
\bigl(f(\rho)-Q_{\flat,\mathscr{M}}^\mu f(\rho)\bigr)
(b-\rho)^{\lambda-1}\,d\rho,
& i=j,
\end{cases}
\]
where \(\rho(x)=b-(b-x)^\mu\). Then there exists a constant
\(C=C(f)\) such that, for every \(x\in\varrho_j\),
\[
\left|
\int_{\rho(x)}^b
\bigl(f(\rho)-Q_{\flat,\mathscr{M}}^\mu f(\rho)\bigr)
(b-\rho)^{\lambda-1}\,d\rho
\right|
\le
\sum_{i=j}^{\mathscr{M}}|\Delta_{j,i;x}|,
\]
and
\[
\sum_{i=j}^{\mathscr{M}}|\Delta_{j,i;x}|
\le
C
\begin{cases}
\varpi\bigl(q(1+\mu\min\{\flat,\gamma\}),\flat+1;\mathscr{M}\bigr),
& \flat\ne\gamma,\\[1mm]
\varpi\bigl(q(\mu\flat+1),\flat+1;\mathscr{M}\bigr)\ln\mathscr{M},
& \flat=\gamma.
\end{cases}
\]
\end{lemma}

\begin{proof}
The first assertion follows immediately from the definition of
\(\Delta_{j,i;x}\) and the triangle inequality:
\[
\left|
\int_{\rho(x)}^b
\bigl(f(\rho)-Q_{\flat,\mathscr{M}}^\mu f(\rho)\bigr)
(b-\rho)^{\lambda-1}\,d\rho
\right|
\le
\sum_{i=j}^{\mathscr{M}}|\Delta_{j,i;x}|.
\]

We first consider the last two mesh intervals. If
\(x\in\varrho_{\mathscr{M}}\), then, by Lemma~\ref{lem:2.2R} and
\(b-x_{\mathscr{M}-1}\le C\mathscr{M}^{-q}\),
\[
\left|
\int_{\rho(x)}^b
\bigl(f(\rho)-Q_{\flat,\mathscr{M}}^\mu f(\rho)\bigr)
(b-\rho)^{\lambda-1}\,d\rho
\right|
=
|\Delta_{\mathscr{M},\mathscr{M};x}|
\le
Ch_{\mathscr{M}}
\|f-Q_{\flat,\mathscr{M}}^\mu f\|_{L^\infty(\varrho_{\mathscr{M},\mu})}
\]
\[
\le
C\mathscr{M}^{-q}
\begin{cases}
\mathscr{M}^{-\min\{q\mu\flat,\flat\}}, & \flat<\gamma,\\[1mm]
\mathscr{M}^{-\min\{q\mu\flat,\flat\}}\ln\mathscr{M}, & \flat=\gamma,\\[1mm]
\mathscr{M}^{-\min\{q\mu\gamma,\flat\}}, & \flat>\gamma.
\end{cases}
\]
Similarly, if \(x\in\varrho_{\mathscr{M}-1}\), then
\(b-x_{\mathscr{M}-2}\le C\mathscr{M}^{-q}\), and
Lemma~\ref{lem:2.2R} gives
\[
\left|
\int_{\rho(x)}^b
\bigl(f(\rho)-Q_{\flat,\mathscr{M}}^\mu f(\rho)\bigr)
(b-\rho)^{\lambda-1}\,d\rho
\right|
\le
|\Delta_{\mathscr{M}-1,\mathscr{M}-1;x}|
+
|\Delta_{\mathscr{M}-1,\mathscr{M};x}|
\]
\[
\le
C\mathscr{M}^{-q}
\begin{cases}
\mathscr{M}^{-\min\{q\mu\flat,\flat\}}, & \flat<\gamma,\\[1mm]
\mathscr{M}^{-\min\{q\mu\flat,\flat\}}\ln\mathscr{M}, & \flat=\gamma,\\[1mm]
\mathscr{M}^{-\min\{q\mu\gamma,\flat\}}, & \flat>\gamma.
\end{cases}
\]

It remains to treat the case \(x\in\varrho_j\),
\(1\le j\le\mathscr{M}-2\). The two end contributions
\(i=j\) and \(i=\mathscr{M}\) are estimated by
Lemma~\ref{lem:2.3R}:
\[
|\Delta_{j,j;x}|+|\Delta_{j,\mathscr{M};x}|
\le
Ch_j\|f-Q_{\flat,\mathscr{M}}^\mu f\|_{L^\infty(\varrho_{j,\mu})}
+
Ch_{\mathscr{M}}
\|f-Q_{\flat,\mathscr{M}}^\mu f\|_{L^\infty(\varrho_{\mathscr{M},\mu})}.
\]
Consequently,
\begin{equation}\label{eq:4.2}
|\Delta_{j,j;x}|+|\Delta_{j,\mathscr{M};x}|
\le
C
\begin{cases}
\mathscr{M}^{-\min\{q(\mu\flat+1),\flat+1\}}, & \flat<\gamma,\\[1mm]
\mathscr{M}^{-\min\{q(\mu\flat+1),\flat+1\}}\ln\mathscr{M},
& \flat=\gamma,\\[1mm]
\mathscr{M}^{-\min\{q(\mu\gamma+1),\flat+1\}}, & \flat>\gamma.
\end{cases}
\end{equation}

We next estimate the interior terms
\(\Delta_{j,i;x}\), \(j<i<\mathscr{M}\). The interpolation error formula gives
\[
\Delta_{j,i;x}
=
\int_{\rho_{i-1}}^{\rho_i}
\bigl(f(\rho)-Q_{\flat,\mathscr{M}}^\mu f(\rho)\bigr)
(b-\rho)^{\lambda-1}\,d\rho
\]
\[
=
\int_{\rho_{i-1}}^{\rho_i}
f[\rho,\eta_{i,1},\ldots,\eta_{i,\flat}]
\prod_{l=1}^{\flat}(\rho-\eta_{i,l})
(b-\rho)^{\lambda-1}\,d\rho,
\]
where \(f[\zeta_1,\ldots,\zeta_{\flat+1}]\) denotes the Newton divided
difference of order \(\flat\). Set
\[
\phi(\rho)
=
f[\rho,\eta_{i,1},\ldots,\eta_{i,\flat}]
(b-\rho)^{\lambda-1}.
\]
By the choice of the collocation parameters in \eqref{eq:4.1R},
\[
\int_{\rho_{i-1}}^{\rho_i}
\prod_{l=1}^{\flat}(\rho-\eta_{i,l})\,d\rho
=
h_{i,\mu}^{\flat+1}
\int_0^1\prod_{l=1}^{\flat}(\zeta-\xi_l)\,d\zeta
=
0.
\]
Taylor's formula then yields
\[
\Delta_{j,i;x}
=
\int_{\rho_{i-1}}^{\rho_i}
\left[
\phi(\rho_{i-1})
+
(\rho-\rho_{i-1})\phi'(\theta_i)
\right]
\prod_{l=1}^{\flat}(\rho-\eta_{i,l})\,d\rho
\]
\begin{equation}\label{eq:4.3}
=
\int_{\rho_{i-1}}^{\rho_i}
(\rho-\rho_{i-1})\phi'(\theta_i)
\prod_{l=1}^{\flat}(\rho-\eta_{i,l})\,d\rho,
\end{equation}
where \(\corr{\theta_i=\theta_i(\rho)}\in(\rho_{i-1},\rho)\).

The standard estimates for divided differences imply that, for some
\(\eta_i^*\in\varrho_{i,\mu}\),
\[
|f[\rho,\eta_{i,1},\ldots,\eta_{i,\flat}]|
=
\left|\frac{f^{(\flat)}(\eta_i^*)}{\flat!}\right|
\le
\max_{\rho\in\varrho_{i,\mu}}|f^{(\flat)}(\rho)|
\]
\begin{equation}\label{eq:4.4}
\le
C
\begin{cases}
1, & \flat<\gamma,\\[1mm]
1+\left|\ln(b-\rho_{i-1})\right|, & \flat=\gamma,\\[1mm]
(b-\rho_{i-1})^{\gamma-\flat}, & \flat>\gamma.
\end{cases}
\end{equation}
Moreover,
\[
\left|
\frac{d}{d\rho}
f[\rho,\eta_{i,1},\ldots,\eta_{i,\flat}]
\right|
=
|f[\rho,\rho,\eta_{i,1},\ldots,\eta_{i,\flat}]|
\le
\max_{\rho\in\varrho_{i,\mu}}|f^{(\flat+1)}(\rho)|
\]
\begin{equation}\label{eq:4.5}
\le
C
\begin{cases}
1, & \flat+1<\gamma,\\[1mm]
1+\left|\ln(b-\rho_{i-1})\right|, & \flat+1=\gamma,\\[1mm]
(b-\rho_{i-1})^{\gamma-\flat-1}, & \flat+1>\gamma.
\end{cases}
\end{equation}
For \(i=j+1,\ldots,\mathscr{M}-1\), the graded mesh satisfies
\(b-x_{i-1}\le C(b-x_i)\). Hence
\[
|\phi'(\theta_i)|
\le
\max_{\rho\in\varrho_{i,\mu}}
\left\{
\left|
\frac{d}{d\rho}f[\rho,\eta_{i,1},\ldots,\eta_{i,\flat}]
\right|
(b-\rho)^{\lambda-1}
+
|\lambda-1|
\left|
f[\rho,\eta_{i,1},\ldots,\eta_{i,\flat}]
\right|
(b-\rho)^{\lambda-2}
\right\}.
\]
Using \eqref{eq:4.4}--\eqref{eq:4.5}, we arrive at
\begin{equation}\label{eq:4.6}
|\phi'(\theta_i)|
\le
C
\begin{cases}
(b-x_{i-1})^{1-2\mu}, & \flat<\gamma,\\[1mm]
(b-x_{i-1})^{1-2\mu}
\bigl(1+|\ln(b-x_{i-1})^\mu|\bigr), & \flat=\gamma,\\[1mm]
(b-x_{i-1})^{\mu(\gamma-\flat)+1-2\mu}, & \flat>\gamma.
\end{cases}
\end{equation}
Combining \eqref{eq:4.3}, \eqref{eq:4.6}, and
Lemma~\ref{lem:2.1R}, we obtain
\[
|\Delta_{j,i;x}|
\le
Ch_{i,\mu}^{\flat+2}
\begin{cases}
(b-x_{i-1})^{1-2\mu}, & \flat<\gamma,\\[1mm]
(b-x_{i-1})^{1-2\mu}
\bigl(1+|\ln(b-x_{i-1})^\mu|\bigr), & \flat=\gamma,\\[1mm]
(b-x_{i-1})^{\mu(\gamma-\flat)+1-2\mu}, & \flat>\gamma.
\end{cases}
\]
Using the mesh estimates, this becomes
\begin{equation}\label{eq:4.7R}
|\Delta_{j,i;x}|
\le
C
\begin{cases}
\mathscr{M}^{-q(\mu\flat+1)}
(1+\mathscr{M}-i)^{q(\mu\flat+1)-\flat-2},
& \flat<\gamma,\\[1mm]
\mathscr{M}^{-q(\mu\flat+1)}
(1+\mathscr{M}-i)^{q(\mu\flat+1)-\flat-2}\ln\mathscr{M},
& \flat=\gamma,\\[1mm]
\mathscr{M}^{-q(\mu\gamma+1)}
(1+\mathscr{M}-i)^{q(\mu\gamma+1)-\flat-2},
& \flat>\gamma.
\end{cases}
\end{equation}

We now sum the interior estimates. First let \(\flat>\gamma\).
From \eqref{eq:4.7R},
\[
\sum_{i=j+1}^{\mathscr{M}-1}|\Delta_{j,i;x}|
\le
C\mathscr{M}^{-q(\mu\gamma+1)}
\sum_{i=j+1}^{\mathscr{M}-1}
(1+\mathscr{M}-i)^{q(\mu\gamma+1)-\flat-2}.
\]
Applying Lemma~\ref{lem:4.1R} with
\(a=0\) and \(t=q(\mu\gamma+1)-\flat-2\) gives
\[
\sum_{i=j+1}^{\mathscr{M}-1}|\Delta_{j,i;x}|
\le
C
\begin{cases}
\mathscr{M}^{-q(\mu\gamma+1)}, & q(\mu\gamma+1)<\flat+1,\\[1mm]
\mathscr{M}^{-\flat-1}\ln\mathscr{M},
& q(\mu\gamma+1)=\flat+1,\\[1mm]
\mathscr{M}^{-\flat-1}, & q(\mu\gamma+1)>\flat+1.
\end{cases}
\]
Together with \eqref{eq:4.2}, this yields
\[
\sum_{i=j}^{\mathscr{M}}|\Delta_{j,i;x}|
\le
C
\begin{cases}
\mathscr{M}^{-q(\mu\gamma+1)}, & q(\mu\gamma+1)<\flat+1,\\[1mm]
\mathscr{M}^{-\flat-1}\ln\mathscr{M},
& q(\mu\gamma+1)=\flat+1,\\[1mm]
\mathscr{M}^{-\flat-1}, & q(\mu\gamma+1)>\flat+1
\end{cases}
=
C\varpi\bigl(q(\mu\gamma+1),\flat+1;\mathscr{M}\bigr).
\]

It remains to consider \(\flat\le\gamma\). Using \eqref{eq:4.7R} and
applying Lemma~\ref{lem:4.1R} with
\(a=0\) and \(t=q(\mu\flat+1)-\flat-2\), we obtain the interior estimate
\[
\sum_{i=j+1}^{\mathscr{M}-1}|\Delta_{j,i;x}|
\le
C
\begin{cases}
\varpi\bigl(q(\mu\flat+1),\flat+1;\mathscr{M}\bigr),
& \flat<\gamma,\\[1mm]
\varpi\bigl(q(\mu\flat+1),\flat+1;\mathscr{M}\bigr)\ln\mathscr{M},
& \flat=\gamma.
\end{cases}
\]
Combining this with the end contribution estimate \eqref{eq:4.2}, we get
\[
\sum_{i=j}^{\mathscr{M}}|\Delta_{j,i;x}|
\le
C
\begin{cases}
\varpi\bigl(q(\mu\flat+1),\flat+1;\mathscr{M}\bigr),
& \flat<\gamma,\\[1mm]
\varpi\bigl(q(\mu\flat+1),\flat+1;\mathscr{M}\bigr)\ln\mathscr{M},
& \flat=\gamma.
\end{cases}
\]
The proof is complete.
\end{proof}

\begin{lemma}\label{lem:4.3R}
Let \(\flat\in\mathbb{N}\) and \(\gamma>0\). Assume that the collocation
parameters satisfy \eqref{eq:4.1R}. Set \(\lambda=1/\mu\), and let
\(f\in C^{\flat+1,1-\gamma}[b-b^\mu,b)\). Then there exists a constant
\(C=C(f)\) such that, for all \(x\in\Lambda\),
\[
\left|
\int_{\rho(x)}^b
\left[a(x)+\mathscr K_1^R(x,\tau(\rho))\right]
\bigl(f(\rho)-Q_{\flat,\mathscr{M}}^\mu f(\rho)\bigr)
(b-\rho)^{\lambda-1}\,d\rho
\right|
\]
\[
\le
C
\begin{cases}
\varpi\bigl(q(1+\mu\min\{\flat,\gamma\}),\flat+1;\mathscr{M}\bigr),
& \flat\ne\gamma,\\[1mm]
\varpi\bigl(q(\mu\flat+1),\flat+1;\mathscr{M}\bigr)\ln\mathscr{M},
& \flat=\gamma.
\end{cases}
\]
\end{lemma}

\begin{proof}
The boundedness of \(a\) on \(\Lambda\) and the estimate
\eqref{eq:K1R-bound} will be used throughout. For
\(x\in\varrho_{\mathscr{M}}\cup\varrho_{\mathscr{M}-1}\),
Lemma~\ref{lem:2.2R} gives
\[
\left|
\int_{\rho(x)}^b
\left[a(x)+\mathscr K_1^R(x,\tau(\rho))\right]
\bigl(f(\rho)-Q_{\flat,\mathscr{M}}^\mu f(\rho)\bigr)
(b-\rho)^{\lambda-1}\,d\rho
\right|
\]
\[
\le
C(b-x_{\mathscr{M}-2})
\|f-Q_{\flat,\mathscr{M}}^\mu f\|_
{L^\infty(\varrho_{\mathscr{M}-1,\mu}\cup\varrho_{\mathscr{M},\mu})}
\]
\[
\le
C\mathscr{M}^{-q}
\begin{cases}
\mathscr{M}^{-\min\{q\mu\flat,\flat\}}, & \flat<\gamma,\\[1mm]
\mathscr{M}^{-\min\{q\mu\flat,\flat\}}\ln\mathscr{M},
& \flat=\gamma,\\[1mm]
\mathscr{M}^{-\min\{q\mu\gamma,\flat\}}, & \flat>\gamma,
\end{cases}
\]
which is stronger than the required bound.

\corr{Now let} \(x\in\varrho_j\), \(1\le j\le\mathscr{M}-2\). We obtain
\begin{equation}\label{eq:4.9R}
\left|
\int_{\rho(x)}^b
\left[a(x)+\mathscr K_1^R(x,\tau(\rho))\right]
\bigl(f(\rho)-Q_{\flat,\mathscr{M}}^\mu f(\rho)\bigr)
(b-\rho)^{\lambda-1}\,d\rho
\right|
\le
\sum_{i=j}^{\mathscr{M}}|\Theta_{j,i;x}|,
\end{equation}
where
\[
\Theta_{j,i;x}
:=
\begin{cases}
\displaystyle
\int_{\rho_{i-1}}^{\rho_i}
\left[a(x)+\mathscr K_1^R(x,\tau(\rho))\right]
\bigl(f(\rho)-Q_{\flat,\mathscr{M}}^\mu f(\rho)\bigr)
(b-\rho)^{\lambda-1}\,d\rho,
& j<i\le\mathscr{M},\\[4mm]
\displaystyle
\int_{\rho(x)}^{\rho_j}
\left[a(x)+\mathscr K_1^R(x,\tau(\rho))\right]
\bigl(f(\rho)-Q_{\flat,\mathscr{M}}^\mu f(\rho)\bigr)
(b-\rho)^{\lambda-1}\,d\rho,
& i=j.
\end{cases}
\]
The terms with \(i=j\), \(i=j+1\), and \(i=\mathscr{M}\) are estimated by
Lemma~\ref{lem:2.3R}. Hence
\begin{equation}\label{eq:4.10R}
|\Theta_{j,j;x}|+|\Theta_{j,j+1;x}|+|\Theta_{j,\mathscr{M};x}|
\le
C
\begin{cases}
\mathscr{M}^{-\min\{q(\mu\flat+1),\flat+1\}}, & \flat<\gamma,\\[1mm]
\mathscr{M}^{-\min\{q(\mu\flat+1),\flat+1\}}\ln\mathscr{M},
& \flat=\gamma,\\[1mm]
\mathscr{M}^{-\min\{q(\mu\gamma+1),\flat+1\}}, & \flat>\gamma.
\end{cases}
\end{equation}

It remains to estimate the middle part
\(i=j+2,\ldots,\mathscr{M}-1\). Write
\[
\Theta_{j,i;x}
=
\Theta_{j,i;x}^{1}+\Theta_{j,i;x}^{2},
\qquad i=j+2,\ldots,\mathscr{M}-1,
\]
where
\[
\Theta_{j,i;x}^{1}
:=
\int_{\rho_{i-1}}^{\rho_i}
\left[a(x)+\mathscr K_1^R(x,x_{i-1})\right]
\bigl(f(\rho)-Q_{\flat,\mathscr{M}}^\mu f(\rho)\bigr)
(b-\rho)^{\lambda-1}\,d\rho,
\]
and
\[
\Theta_{j,i;x}^{2}
:=
\int_{\rho_{i-1}}^{\rho_i}
\left[
\mathscr K_1^R(x,\tau(\rho))
-\mathscr K_1^R(x,x_{i-1})
\right]
\bigl(f(\rho)-Q_{\flat,\mathscr{M}}^\mu f(\rho)\bigr)
(b-\rho)^{\lambda-1}\,d\rho.
\]
Since \(a(x)+\mathscr K_1^R(x,x_{i-1})\) is independent of \(\rho\) on
\(\varrho_{i,\mu}\), Lemma~\ref{lem:4.2R} gives
\begin{equation}\label{eq:4.11R}
\sum_{i=j+2}^{\mathscr{M}-1}|\Theta_{j,i;x}^{1}|
\le
C
\begin{cases}
\varpi\bigl(q(1+\mu\min\{\flat,\gamma\}),\flat+1;\mathscr{M}\bigr),
& \flat\ne\gamma,\\[1mm]
\varpi\bigl(q(\mu\flat+1),\flat+1;\mathscr{M}\bigr)\ln\mathscr{M},
& \flat=\gamma.
\end{cases}
\end{equation}

For the remaining part, using the definition of \(\mathscr K_1^R\) in
\eqref{eq:K1R-def}, we have
\begin{align}
\sum_{i=j+2}^{\mathscr{M}-1}|\Theta_{j,i;x}^{2}|
&\le
C\sum_{i=j+2}^{\mathscr{M}-1}
\|f-Q_{\flat,\mathscr{M}}^\mu f\|_{L^\infty(\varrho_{i,\mu})}
\nonumber\\
&\quad\times
\int_{\rho_{i-1}}^{\rho_i}
\int_{x_{i-1}}^{\tau(\rho)}
(\zeta-x)^{\nu-1}(b-\rho)^{\lambda-1}\,d\zeta\,d\rho .
\label{eq:4.12R}
\end{align}
For \(i\ge j+2\), we have
\(\zeta\ge x_{i-1}>x_j\ge x\). Since \(0<\nu<1\),
\[
(\zeta-x)^{\nu-1}
\le
(x_{i-1}-x_j)^{\nu-1}.
\]
Therefore,
\begin{equation}\label{eq:4.13R}
\int_{\rho_{i-1}}^{\rho_i}
\int_{x_{i-1}}^{\tau(\rho)}
(\zeta-x)^{\nu-1}(b-\rho)^{\lambda-1}\,d\zeta\,d\rho
\le
C(x_{i-1}-x_j)^{\nu-1}h_i^2.
\end{equation}
Moreover, since \(h_i\le h_j\) for \(i\ge j\),
\[
x_{i-1}-x_j
=
\sum_{k=j+1}^{i-1}h_k
\ge
(i-j-1)h_i.
\]
Combining this with \eqref{eq:4.12R} and \eqref{eq:4.13R}, we obtain
\begin{equation}\label{eq:4.14R}
\sum_{i=j+2}^{\mathscr{M}-1}|\Theta_{j,i;x}^{2}|
\le
C
\sum_{i=j+2}^{\mathscr{M}-1}
h_i^{1+\nu}(i-j-1)^{\nu-1}
\|f-Q_{\flat,\mathscr{M}}^\mu f\|_{L^\infty(\varrho_{i,\mu})}.
\end{equation}
By the mesh definition and the \corr{mean value theorem},
\[
h_i^{1+\nu}
\le
C\mathscr{M}^{-q(1+\nu)}
(1+\mathscr{M}-i)^{(q-1)(1+\nu)}.
\]
Thus Lemmas~\ref{lem:2.1R} and~\ref{lem:2.2R} imply
\[
\sum_{i=j+2}^{\mathscr{M}-1}|\Theta_{j,i;x}^{2}|
\le
C\mathscr{M}^{-q(1+\nu)}
\sum_{i=j+2}^{\mathscr{M}-1}
(i-j-1)^{\nu-1}
\]
\[
\times
\begin{cases}
\mathscr{M}^{-q\mu\flat}
(1+\mathscr{M}-i)^{q\mu\flat-\flat+(q-1)(1+\nu)},
& \flat<\gamma,\\[1mm]
\mathscr{M}^{-q\mu\flat}
(1+\mathscr{M}-i)^{q\mu\flat-\flat+(q-1)(1+\nu)}
\ln\mathscr{M},
& \flat=\gamma,\\[1mm]
\mathscr{M}^{-q\mu\gamma}
(1+\mathscr{M}-i)^{q\mu\gamma-\flat+(q-1)(1+\nu)},
& \flat>\gamma.
\end{cases}
\]

Consider first \(\flat>\gamma\). Applying Lemma~\ref{lem:4.1R} with
\(a=\nu-1\) and
\(t=q\mu\gamma-\flat+(q-1)(1+\nu)\)
to the last sum gives
\[
\sum_{i=j+2}^{\mathscr{M}-1}|\Theta_{j,i;x}^{2}|
\le
C\mathscr{M}^{-q(\mu\gamma+1+\nu)}
\begin{cases}
(\mathscr{M}-j)^{\nu-1},
& q(\mu\gamma+1+\nu)<\flat+\nu,\\[1mm]
(\mathscr{M}-j)^{\nu-1}\ln(\mathscr{M}-j+1),
& q(\mu\gamma+1+\nu)=\flat+\nu,\\[1mm]
(\mathscr{M}-j)^{q(\mu\gamma+1+\nu)-\flat-1},
& q(\mu\gamma+1+\nu)>\flat+\nu.
\end{cases}
\]
Consequently,
\begin{equation}\label{eq:4.15R}
\sum_{i=j+2}^{\mathscr{M}-1}|\Theta_{j,i;x}^{2}|
\le
C
\begin{cases}
\mathscr{M}^{-q(\mu\gamma+1+\nu)},
& q(\mu\gamma+1)\le\flat+1-q\nu,\\[1mm]
\mathscr{M}^{-\flat-1},
& q(\mu\gamma+1)>\flat+1-q\nu.
\end{cases}
\end{equation}
On the other hand, \eqref{eq:4.11R} gives, for \(\flat>\gamma\),
\begin{equation}\label{eq:4.16R}
\sum_{i=j+2}^{\mathscr{M}-1}|\Theta_{j,i;x}^{1}|
\le
C
\begin{cases}
\mathscr{M}^{-q(\mu\gamma+1)},
& q(\mu\gamma+1)<\flat+1,\\[1mm]
\mathscr{M}^{-\flat-1}\ln\mathscr{M},
& q(\mu\gamma+1)=\flat+1,\\[1mm]
\mathscr{M}^{-\flat-1},
& q(\mu\gamma+1)>\flat+1.
\end{cases}
\end{equation}
\corr{Since the bound in \eqref{eq:4.15R} is no larger than that in
\eqref{eq:4.16R},} it follows that
\begin{equation}\label{eq:4.17R}
\sum_{i=j+2}^{\mathscr{M}-1}|\Theta_{j,i;x}|
\le
C
\begin{cases}
\mathscr{M}^{-q(\mu\gamma+1)},
& q(\mu\gamma+1)<\flat+1,\\[1mm]
\mathscr{M}^{-\flat-1}\ln\mathscr{M},
& q(\mu\gamma+1)=\flat+1,\\[1mm]
\mathscr{M}^{-\flat-1},
& q(\mu\gamma+1)>\flat+1.
\end{cases}
\end{equation}
Combining \eqref{eq:4.10R} and \eqref{eq:4.17R}, we get
\[
\sum_{i=j}^{\mathscr{M}}|\Theta_{j,i;x}|
\le
C\varpi(q(\mu\gamma+1),\flat+1;\mathscr{M}),
\qquad \flat>\gamma.
\]

The cases \(\flat\le\gamma\) follow from the same argument. Hence
\[
\sum_{i=j}^{\mathscr{M}}|\Theta_{j,i;x}|
\le
C
\begin{cases}
\varpi(q(\mu\flat+1),\flat+1;\mathscr{M}),
& \flat<\gamma,\\[1mm]
\varpi(q(\mu\flat+1),\flat+1;\mathscr{M})\ln\mathscr{M},
& \flat=\gamma.
\end{cases}
\]
The assertion follows from \eqref{eq:4.9R}.
\end{proof}

\section{Error analysis of the collocation solution}
\label{sec5}

In this section, we establish error estimates for the fractional backward
collocation solution determined by \eqref{eq:3.1aR}--\eqref{eq:3.1bR}. The
analysis is based on the Volterra reformulation
\eqref{eq:v-equation-right-regularity}, where the kernel \(\mathscr K_1^R\)
is defined in \eqref{eq:K1R-def}. By \eqref{eq:K1R-bound},
\(\mathscr K_1^R\) is bounded on \(D_R\).

Define the linear Volterra operator
\(\mathcal V:L^\infty(\Lambda)\to L^\infty(\Lambda)\) by
\begin{equation}\label{eq:5.KopR}
(\mathcal V f)(x)
:=
-\int_x^b
\left[a(x)+\mathscr K_1^R(x,s)\right]f(s)\,ds,
\qquad x\in\Lambda .
\end{equation}
Since \(a\in C(\Lambda)\) and \(\mathscr K_1^R\) is continuous and bounded on
\(D_R\), the operator \(\mathcal V\) \corr{maps} bounded subsets of
\(L^\infty(\Lambda)\) into uniformly bounded and equicontinuous subsets of
\(C(\Lambda)\). Consequently, 
\(\mathcal V\) is compact as an operator from
\(L^\infty(\Lambda)\) into \(L^\infty(\Lambda)\).

Let \(\flat\ge2\). We employ the graded mesh
\(\mathscr{P}_{\mathscr{M}}\), the collocation points \eqref{eq:2.1R}, and
the fractional interpolation operator
\(I_{\flat,\mathscr{M}}^{\mu}:C(\Lambda)\to
S_\flat^\mu(\mathscr{P}_{\mathscr{M}})\)
introduced in Section~\ref{sec2}. Since \(\mathcal V f\in C(\Lambda)\) for
all \(f\in L^\infty(\Lambda)\), the operator
\(I_{\flat,\mathscr{M}}^{\mu}\mathcal V\) is well defined on
\(L^\infty(\Lambda)\). Let \(\mathcal I\) be the identity operator. We shall
estimate
\[
\mathcal I-I_{\flat,\mathscr{M}}^{\mu}\mathcal V:
L^\infty(\Lambda)\to L^\infty(\Lambda).
\]
For a bounded linear operator
\(S:L^\infty(\Lambda)\to L^\infty(\Lambda)\), we denote its operator norm by
\(\|S\|_{L^\infty(\Lambda)\to L^\infty(\Lambda)}\).

\begin{lemma}\label{lem:5.1R}
There exist \(\mathscr{M}_0\in\mathbb N\) and a constant \(C>0\), independent of
\(\mathscr{M}\), such that, for all \(\mathscr{M}>\mathscr{M}_0\),
\(\left(\mathcal I-I_{\flat,\mathscr{M}}^{\mu}\mathcal V\right)^{-1}:
L^\infty(\Lambda)\to L^\infty(\Lambda)\)
exists and
\[
\sup_{\mathscr{M}>\mathscr{M}_0}
\left\|
\left(\mathcal I-I_{\flat,\mathscr{M}}^{\mu}\mathcal V\right)^{-1}
\right\|_{L^\infty(\Lambda)\to L^\infty(\Lambda)}
\le C .
\]
\end{lemma}

\begin{proof}
We first verify that \(\mathcal I-\mathcal V\) is injective. Let
\(z\in L^\infty(\Lambda)\) satisfy
\[
(\mathcal I-\mathcal V)z=0 .
\]
By the definition of \(\mathcal V\) in \eqref{eq:5.KopR},
\[
z(x)
=
-\int_x^b
\left[a(x)+\mathscr K_1^R(x,s)\right]z(s)\,ds,
\qquad x\in\Lambda .
\]
Since \(a\in C(\Lambda)\) and \(\mathscr K_1^R\) satisfies
\eqref{eq:K1R-bound}, there is a constant \(C_1>0\) such that
\[
|a(x)|+|\mathscr K_1^R(x,s)|\le C_1,
\qquad 0\le x\le s\le b .
\]
Thus
\[
|z(x)|
\le
C_1\int_x^b|z(s)|\,ds,
\qquad x\in\Lambda .
\]
Choose \(\corr{\delta:=\min\{b,(2C_1)^{-1}\}}\). Then
\[
\|z\|_{L^\infty(b-\delta,b)}
\le
C_1\delta\,
\|z\|_{L^\infty(b-\delta,b)}
\le
\frac12
\|z\|_{L^\infty(b-\delta,b)}.
\]
Hence \(z=0\) on \([b-\delta,b]\). Repeating the same argument on successive
intervals of length \(\delta\), moving backward from \(b\) to \(0\), gives
\(z=0\) on \(\Lambda\). Therefore, the homogeneous equation has only the trivial
solution.

Since \(\mathcal V\) is compact on \(L^\infty(\Lambda)\), the Fredholm alternative
implies that
\[
\mathcal I-\mathcal V:L^\infty(\Lambda)\to L^\infty(\Lambda)
\]
is invertible. \corr{Moreover,} the compactness of \(\mathcal V\), together with the
uniform convergence of the interpolation operator on compact subsets of
\(C(\Lambda)\); see \cite[Lemma~3.2]{brunner2001piecewise}, gives
\[
\lim_{\mathscr{M}\to\infty}
\left\|
\mathcal V-I_{\flat,\mathscr{M}}^{\mu}\mathcal V
\right\|_{L^\infty(\Lambda)\to L^\infty(\Lambda)}
=0 .
\]
Consequently, by the \corr{approximation result}
\cite[Theorem~3.1.1]{Atkinson_1997}, there exist
\(\mathscr{M}_0\in\mathbb{N}\) and a constant \(C>0\), independent of
\(\mathscr{M}\), such that, for all \(\mathscr{M}>\mathscr{M}_0\), the operator
\[
\left(\mathcal I-I_{\flat,\mathscr{M}}^{\mu}\mathcal V\right)^{-1}
:
L^\infty(\Lambda)\to L^\infty(\Lambda)
\]
exists and satisfies
\[
\sup_{\mathscr{M}>\mathscr{M}_0}
\left\|
\left(\mathcal I-I_{\flat,\mathscr{M}}^{\mu}\mathcal V\right)^{-1}
\right\|_{L^\infty(\Lambda)\to L^\infty(\Lambda)}
\le C .
\]
The proof is complete.
\end{proof}

We first establish the convergence estimate for arbitrary collocation
parameters \(0\le\xi_1<\cdots<\xi_\flat\le1\).

\subsection{Error estimates for general collocation parameters}

\begin{theorem}\label{thm:5.1R}
Suppose that
\[
\widehat{\psi}(\rho):=\psi(\tau(\rho))
\in C^{\flat,1-\gamma}[b-b^\mu,b).
\]
Then there exist an integer \(\mathscr{M}_0\in\mathbb N\) and a constant
\(C\), independent of \(\mathscr{M}\), such that, whenever
\(\mathscr{M}>\mathscr{M}_0\), the collocation approximations
\(\psi_h\) and \(\varphi_h\), defined by
\eqref{eq:3.1aR}--\eqref{eq:3.1bR}, satisfy
\begin{equation}\label{eq:5.1R}
\|\psi-\psi_h\|_{L^\infty(\Lambda)}
\le
C
\begin{cases}
\mathscr{M}^{-\min\{q\mu\flat,\flat\}}, & \flat<\gamma,\\[1mm]
\mathscr{M}^{-\min\{q\mu\flat,\flat\}}\ln\mathscr{M}, & \flat=\gamma,\\[1mm]
\mathscr{M}^{-\min\{q\mu\gamma,\flat\}}, & \flat>\gamma,
\end{cases}
\end{equation}
and
\begin{equation}\label{eq:5.2R}
\|\varphi-\varphi_h\|_{L^\infty(\Lambda)}
\le
C
\begin{cases}
\varpi\bigl(q(1+\mu\min\{\flat,\gamma\}),\flat;\mathscr{M}\bigr),
& \flat\ne\gamma,\\[1mm]
\varpi\bigl(q(\mu\flat+1),\flat;\mathscr{M}\bigr)\ln\mathscr{M},
& \flat=\gamma.
\end{cases}
\end{equation}
\end{theorem}

\begin{proof}
By the Volterra reformulation \eqref{eq:v-equation-right-regularity}, the
exact derivative \(\psi=\varphi'\) satisfies
\begin{equation}\label{eq:5.3R}
\psi=\chi_R+\mathcal V\psi .
\end{equation}
The collocation equations impose the same relation at the points \(x_{j,l}\).
Since \(\psi_h\in S_\flat^\mu(\mathscr{P}_{\mathscr{M}})\), this gives
\[
\psi_h
=
I_{\flat,\mathscr{M}}^{\mu}\chi_R
+
I_{\flat,\mathscr{M}}^{\mu}\mathcal V\psi_h .
\]
Applying \(I_{\flat,\mathscr{M}}^{\mu}\) to \eqref{eq:5.3R}, we also have
\[
I_{\flat,\mathscr{M}}^{\mu}\psi
=
I_{\flat,\mathscr{M}}^{\mu}\chi_R
+
I_{\flat,\mathscr{M}}^{\mu}\mathcal V\psi .
\]
Subtracting the two identities and writing
\[
\psi-\psi_h
=
\psi-I_{\flat,\mathscr{M}}^{\mu}\psi
+
I_{\flat,\mathscr{M}}^{\mu}\psi-\psi_h,
\]
we obtain
\begin{equation}\label{eq:5.4R}
\left(\mathcal I-I_{\flat,\mathscr{M}}^{\mu}\mathcal V\right)
\left(I_{\flat,\mathscr{M}}^{\mu}\psi-\psi_h\right)
=
I_{\flat,\mathscr{M}}^{\mu}\mathcal V
\left(\psi-I_{\flat,\mathscr{M}}^{\mu}\psi\right).
\end{equation}
Let \(\mathscr{M}_0\) be chosen as in Lemma~\ref{lem:5.1R}. For
\(\mathscr{M}>\mathscr{M}_0\), \eqref{eq:5.4R} yields
\begin{equation}\label{eq:5.5R}
I_{\flat,\mathscr{M}}^{\mu}\psi-\psi_h
=
\left(\mathcal I-I_{\flat,\mathscr{M}}^{\mu}\mathcal V\right)^{-1}
I_{\flat,\mathscr{M}}^{\mu}\mathcal V
\left(\psi-I_{\flat,\mathscr{M}}^{\mu}\psi\right).
\end{equation}
Using Lemma~\ref{lem:5.1R} and the uniform boundedness of
\(I_{\flat,\mathscr{M}}^{\mu}\), we get
\[
\left\|
I_{\flat,\mathscr{M}}^{\mu}\psi-\psi_h
\right\|_{L^\infty(\Lambda)}
\le
C
\left\|
\mathcal V
\left(\psi-I_{\flat,\mathscr{M}}^{\mu}\psi\right)
\right\|_{L^\infty(\Lambda)} .
\]
By the definition of \(\mathcal V\), the change of variable
\(s=\tau(\rho)\), and \eqref{eq:2.3R},
\[
\psi(s)-I_{\flat,\mathscr{M}}^{\mu}\psi(s)
=
\widehat{\psi}(\rho)
-
Q_{\flat,\mathscr{M}}^{\mu}\widehat{\psi}(\rho),
\qquad s=\tau(\rho).
\]
Thus, with \(\lambda=1/\mu\), \corr{and absorbing the constant factor
\(\lambda\) into \(C\),} we obtain
\begin{equation}\label{eq:5.6R}
\begin{aligned}
&
\left\|
I_{\flat,\mathscr{M}}^{\mu}\psi-\psi_h
\right\|_{L^\infty(\Lambda)}
\\
&\quad\le
C
\left\|
\int_{\rho(x)}^b
\left[a(x)+\mathscr K_1^R(x,\tau(\rho))\right]
\left(
\widehat{\psi}(\rho)
-
Q_{\flat,\mathscr{M}}^{\mu}\widehat{\psi}(\rho)
\right)
(b-\rho)^{\lambda-1}\,d\rho
\right\|_{L^\infty(\Lambda)} .
\end{aligned}
\end{equation}
The boundedness of \(a\) on \(\Lambda\) and the estimate
\eqref{eq:K1R-bound} imply
\[
\left\|
I_{\flat,\mathscr{M}}^{\mu}\psi-\psi_h
\right\|_{L^\infty(\Lambda)}
\le
C
\left\|
\widehat{\psi}
-
Q_{\flat,\mathscr{M}}^{\mu}\widehat{\psi}
\right\|_{L^\infty[b-b^\mu,b]} .
\]
Consequently, by the triangle inequality, \eqref{eq:2.3R}, and
Lemma~\ref{lem:2.2R},
\[
\begin{aligned}
\|\psi-\psi_h\|_{L^\infty(\Lambda)}
&\le
\|\psi-I_{\flat,\mathscr{M}}^{\mu}\psi\|_{L^\infty(\Lambda)}
+
\|I_{\flat,\mathscr{M}}^{\mu}\psi-\psi_h\|_{L^\infty(\Lambda)}
\\
&\le
C
\left\|
\widehat{\psi}
-
Q_{\flat,\mathscr{M}}^{\mu}\widehat{\psi}
\right\|_{L^\infty[b-b^\mu,b]}
\\
&\le
C
\begin{cases}
\mathscr{M}^{-\min\{q\mu\flat,\flat\}}, & \flat<\gamma,\\[1mm]
\mathscr{M}^{-\min\{q\mu\flat,\flat\}}\ln\mathscr{M},
& \flat=\gamma,\\[1mm]
\mathscr{M}^{-\min\{q\mu\gamma,\flat\}}, & \flat>\gamma.
\end{cases}
\end{aligned}
\]
This proves \eqref{eq:5.1R}.

It remains to estimate \(\varphi-\varphi_h\). From the terminal
representations of \(\varphi\) and \(\varphi_h\),
\[
\varphi(x)-\varphi_h(x)
=
-\int_x^b\bigl(\psi(s)-\psi_h(s)\bigr)\,ds .
\]
\corr{Using \eqref{eq:5.5R}, Lemma~\ref{lem:5.1R}, and the Volterra
structure of the resolvent,} we obtain
\[
|\varphi(x)-\varphi_h(x)|
\le
C
\int_x^b
\left|
\psi(s)-I_{\flat,\mathscr{M}}^{\mu}\psi(s)
\right|\,ds .
\]
For \(x\in\varrho_j\), splitting the last integral over the mesh intervals gives
\[
|\varphi(x)-\varphi_h(x)|
\le
C
\sum_{i=j}^{\mathscr{M}}
h_i
\left\|
\psi-I_{\flat,\mathscr{M}}^{\mu}\psi
\right\|_{L^\infty(\varrho_i)} .
\]
Using \eqref{eq:2.3R}, this becomes
\[
|\varphi(x)-\varphi_h(x)|
\le
C
\sum_{i=j}^{\mathscr{M}}
h_i
\left\|
\widehat{\psi}
-
Q_{\flat,\mathscr{M}}^{\mu}\widehat{\psi}
\right\|_{L^\infty(\varrho_{i,\mu})}.
\]
\corr{Using the intervalwise estimate in the proof of
Lemma~\ref{lem:2.3R},} we obtain
\[
|\varphi(x)-\varphi_h(x)|
\le
C
\begin{cases}
\mathscr{M}^{-q(\mu\flat+1)}
\displaystyle
\sum_{i=j}^{\mathscr{M}}
(1+\mathscr{M}-i)^{q(\mu\flat+1)-\flat-1},
& \flat<\gamma,\\[4mm]
\mathscr{M}^{-q(\mu\flat+1)}\ln\mathscr{M}
\displaystyle
\sum_{i=j}^{\mathscr{M}}
(1+\mathscr{M}-i)^{q(\mu\flat+1)-\flat-1},
& \flat=\gamma,\\[4mm]
\mathscr{M}^{-q(\mu\gamma+1)}
\displaystyle
\sum_{i=j}^{\mathscr{M}}
(1+\mathscr{M}-i)^{q(\mu\gamma+1)-\flat-1},
& \flat>\gamma.
\end{cases}
\]
For any \(\alpha\in\mathbb R\), the elementary bound
\[
\sum_{i=j}^{\mathscr{M}}
(1+\mathscr{M}-i)^{\alpha}
\le
C
\begin{cases}
1, & \alpha<-1,\\
\ln\mathscr{M}, & \alpha=-1,\\
\mathscr{M}^{\alpha+1}, & \alpha>-1
\end{cases}
\]
gives
\[
|\varphi(x)-\varphi_h(x)|
\le
C\varpi\bigl(q(\mu\flat+1),\flat;\mathscr{M}\bigr),
\qquad \flat<\gamma,
\]
\[
|\varphi(x)-\varphi_h(x)|
\le
C\varpi\bigl(q(\mu\flat+1),\flat;\mathscr{M}\bigr)\ln\mathscr{M},
\qquad \flat=\gamma,
\]
and
\[
|\varphi(x)-\varphi_h(x)|
\le
C\varpi\bigl(q(\mu\gamma+1),\flat;\mathscr{M}\bigr),
\qquad \flat>\gamma.
\]
Therefore,
\[
\|\varphi-\varphi_h\|_{L^\infty(\Lambda)}
\le
C
\begin{cases}
\varpi\bigl(q(1+\mu\min\{\flat,\gamma\}),\flat;\mathscr{M}\bigr),
& \flat\ne\gamma,\\[1mm]
\varpi\bigl(q(\mu\flat+1),\flat;\mathscr{M}\bigr)\ln\mathscr{M},
& \flat=\gamma.
\end{cases}
\]
This proves \eqref{eq:5.2R}. The proof is complete.
\end{proof}

\subsection{Superconvergent estimates under \corr{condition~\eqref{eq:4.1R}}}
\label{subsec:superconvergenceR}

\corr{When the collocation parameters \(\{\xi_k\}_{k=1}^{\flat}\) satisfy
\eqref{eq:4.1R}, the preceding error estimates can be sharpened as follows.}

\begin{theorem}\label{thm:5.2R}
Assume that \(\mathscr K\in C^1(D_R)\) and
\[
\widehat{\psi}(\rho):=\psi(\tau(\rho))
\in C^{\flat+1,1-\gamma}[b-b^\mu,b).
\]
Let the collocation parameters \(\corr{\{\xi_k\}_{k=1}^{\flat}}\) satisfy
\eqref{eq:4.1R}. Then there exist \(\mathscr{M}_0\in\mathbb N\) and a constant
\(C\), independent of \(\mathscr{M}\), such that, for all
\(\mathscr{M}>\mathscr{M}_0\), the collocation
\corr{approximations} \(\psi_h,\varphi_h\) defined by
\eqref{eq:3.1aR}--\eqref{eq:3.1bR} satisfy
\begin{equation}\label{eq:5.7R}
\|\varphi-\varphi_h\|_{L^\infty(\Lambda)}
\le
C
\begin{cases}
\varpi\bigl(q(1+\mu\min\{\flat,\gamma\}),\flat+1;\mathscr{M}\bigr),
& \flat\ne\gamma,\\[1mm]
\varpi\bigl(q(\mu\flat+1),\flat+1;\mathscr{M}\bigr)\ln\mathscr{M},
& \flat=\gamma,
\end{cases}
\end{equation}
\corr{while, at the collocation points,}
\begin{equation}\label{eq:5.8R}
\max_{x\in\cup_{j=1}^{\mathscr{M}}X_j}
|\psi(x)-\psi_h(x)|
\le
C
\begin{cases}
\varpi\bigl(q(1+\mu\min\{\flat,\gamma\}),\flat+1;\mathscr{M}\bigr),
& \flat\ne\gamma,\\[1mm]
\varpi\bigl(q(\mu\flat+1),\flat+1;\mathscr{M}\bigr)\ln\mathscr{M},
& \flat=\gamma,
\end{cases}
\end{equation}
and, at \(x=0\),
\begin{equation}\label{eq:5.9R}
|\psi(0)-\psi_h(0)|
\le
C
\begin{cases}
\varpi\bigl(q(1+\mu\min\{\flat,\gamma\}),\flat;\mathscr{M}\bigr),
& \flat\ne\gamma,\\[1mm]
\varpi\bigl(q(\mu\flat+1),\flat;\mathscr{M}\bigr)\ln\mathscr{M},
& \flat=\gamma.
\end{cases}
\end{equation}
\end{theorem}

\begin{proof}
For \(x\in\Lambda\), the terminal representations of \(\varphi\) and
\(\varphi_h\) give
\[
|\varphi(x)-\varphi_h(x)|
\le
\left|
\int_x^b
\bigl(\psi(s)-I_{\flat,\mathscr{M}}^{\mu}\psi(s)\bigr)\,ds
\right|
+
b
\left\|
I_{\flat,\mathscr{M}}^{\mu}\psi-\psi_h
\right\|_{L^\infty(\Lambda)} .
\]
Using the change of variables \(s=\tau(\rho)\), together with
\eqref{eq:2.3R}, the first term \corr{can be written as}
\[
\corr{\frac{1}{\mu}}
\left|
\int_{\rho(x)}^b
\bigl(
\widehat{\psi}(\rho)
-
Q_{\flat,\mathscr{M}}^{\mu}\widehat{\psi}(\rho)
\bigr)
(b-\rho)^{1/\mu-1}\,d\rho
\right| .
\]
Hence Lemma~\ref{lem:4.2R}, applied to \(f=\widehat{\psi}\),
\corr{with the factor \(1/\mu\) absorbed into \(C\),} gives
\[
\left|
\int_x^b
\bigl(\psi(s)-I_{\flat,\mathscr{M}}^{\mu}\psi(s)\bigr)\,ds
\right|
\le
C
\begin{cases}
\varpi\bigl(q(1+\mu\min\{\flat,\gamma\}),\flat+1;\mathscr{M}\bigr),
& \flat\ne\gamma,\\[1mm]
\varpi\bigl(q(\mu\flat+1),\flat+1;\mathscr{M}\bigr)\ln\mathscr{M},
& \flat=\gamma.
\end{cases}
\]
On the other hand, \eqref{eq:5.6R} and Lemma~\ref{lem:4.3R} imply
\begin{equation}\label{eq:5.10R}
\left\|
I_{\flat,\mathscr{M}}^{\mu}\psi-\psi_h
\right\|_{L^\infty(\Lambda)}
\le
C
\begin{cases}
\varpi\bigl(q(1+\mu\min\{\flat,\gamma\}),\flat+1;\mathscr{M}\bigr),
& \flat\ne\gamma,\\[1mm]
\varpi\bigl(q(\mu\flat+1),\flat+1;\mathscr{M}\bigr)\ln\mathscr{M},
& \flat=\gamma.
\end{cases}
\end{equation}
Combining the last two estimates proves \eqref{eq:5.7R}.

We next estimate the error at the collocation points. Let
\(x\in\cup_{j=1}^{\mathscr{M}}X_j\). Subtracting the collocation equation
\eqref{eq:3.3R-coll} from the exact equation \eqref{eq:1.1R} at such a point
gives
\[
\begin{aligned}
|\psi(x)-\psi_h(x)|
&\le
|a(x)|\,|\varphi(x)-\varphi_h(x)|
\\
&\quad+
\int_x^b
(s-x)^{\nu-1}|\mathscr K(x,s)|
\,|\varphi(s)-\varphi_h(s)|\,ds .
\end{aligned}
\]
Since \(a\) and \(\mathscr K\) are bounded and \(0<\nu<1\), we obtain
\[
|\psi(x)-\psi_h(x)|
\le
C\|\varphi-\varphi_h\|_{L^\infty(\Lambda)} .
\]
Taking the maximum over the collocation points and using \eqref{eq:5.7R}
proves \eqref{eq:5.8R}.

It remains to estimate the error at \(x=0\). Since
\(\corr{0=x_0\in\overline{\varrho}_1}\), \corr{the continuity of the
interpolation error on the first mesh interval yields}
\[
|\psi(0)-\psi_h(0)|
\le
\left\|
\psi-I_{\flat,\mathscr{M}}^{\mu}\psi
\right\|_{L^\infty(\varrho_1)}
+
\left\|
I_{\flat,\mathscr{M}}^{\mu}\psi-\psi_h
\right\|_{L^\infty(\Lambda)} .
\]
By \eqref{eq:2.3R} and \eqref{eq:2.4R},
\[
\left\|
\psi-I_{\flat,\mathscr{M}}^{\mu}\psi
\right\|_{L^\infty(\varrho_1)}
=
\left\|
\widehat{\psi}
-
Q_{\flat,\mathscr{M}}^{\mu}\widehat{\psi}
\right\|_{L^\infty(\varrho_{1,\mu})}
\le
C\mathscr{M}^{-\flat}.
\]
Combining this estimate with \eqref{eq:5.10R} gives
\[
|\psi(0)-\psi_h(0)|
\le
C
\begin{cases}
\varpi\bigl(q(1+\mu\min\{\flat,\gamma\}),\flat;\mathscr{M}\bigr),
& \flat\ne\gamma,\\[1mm]
\varpi\bigl(q(\mu\flat+1),\flat;\mathscr{M}\bigr)\ln\mathscr{M},
& \flat=\gamma.
\end{cases}
\]
This proves \eqref{eq:5.9R}. The proof is complete.
\end{proof}

\section{The choice of \(\mu\)}
\label{subsec:choice-lambdaR}

The estimates in Theorems~\ref{thm:5.1R} and~\ref{thm:5.2R} are stated in
terms of the transformed function
\[
\widehat{\psi}(\rho)=\psi(\tau(\rho))\in
C^{\corr{\ell},1-\gamma}[b-b^\mu,b),
\qquad
\corr{\ell\in\{\flat,\flat+1\}}.
\]
Thus, the choice of the fractional parameter \(\mu\) should reflect the
terminal expansion of the exact solution. The next result makes this choice
explicit in terms of the singularity exponent \(\nu\).

\begin{corollary}\label{cor:5.1R}
Let \(\mu\in(0,1)\) be chosen so that \(\nu/\mu\in\mathbb N\). Then
the collocation \corr{approximations} \(\psi_h,\varphi_h\) defined by
\eqref{eq:3.1aR}--\eqref{eq:3.1bR} satisfy the following estimates.

\medskip
\noindent{\rm (i)}
Assume that
\[
a,\chi\in C^\flat(\Lambda),\qquad
\mathscr K\in C^\flat(D_R),\qquad
\mathscr K(x,x)\ne0\quad \corr{\text{for all }x\in\Lambda}.
\]
Then there exist \(\mathscr{M}_0\in\mathbb N\) and a constant \(C\), independent of
\(\mathscr{M}\), such that, for all \(\mathscr{M}>\mathscr{M}_0\),
\begin{equation}\label{eq:cor51-vR}
\|\psi-\psi_h\|_{L^\infty(\Lambda)}
\le
C
\begin{cases}
\mathscr{M}^{-\min\{q\mu\flat,\flat\}},
& 1/\mu\in\mathbb N \ \text{or}\ 1/\mu>\flat,\\[1mm]
\mathscr{M}^{-\min\{q,\flat\}},
& 1/\mu\notin\mathbb N \ \text{and}\ 1/\mu<\flat,
\end{cases}
\end{equation}
and
\begin{equation}\label{eq:cor51-uR}
\|\varphi-\varphi_h\|_{L^\infty(\Lambda)}
\le
C
\begin{cases}
\varpi\bigl(q(\mu\flat+1),\flat;\mathscr{M}\bigr),
& 1/\mu\in\mathbb N \ \text{or}\ 1/\mu>\flat,\\[1mm]
\varpi\bigl(2q,\flat;\mathscr{M}\bigr),
& 1/\mu\notin\mathbb N \ \text{and}\ 1/\mu<\flat.
\end{cases}
\end{equation}

\medskip
\noindent{\rm (ii)}
Assume that
\[
a,\chi\in C^{\flat+1}(\Lambda),\qquad
\mathscr K\in C^{\flat+1}(D_R),\qquad
\mathscr K(x,x)\ne0\quad \corr{\text{for all }x\in\Lambda}.
\]
Assume also that the collocation parameters
\(\corr{\{\xi_k\}_{k=1}^{\flat}}\) satisfy \eqref{eq:4.1R}. Then there exist
\(\mathscr{M}_0\in\mathbb N\) and a constant \(C\), independent of
\(\mathscr{M}\), such that, for all \(\mathscr{M}>\mathscr{M}_0\),
\begin{equation}\label{eq:cor51-u-superR}
\|\varphi-\varphi_h\|_{L^\infty(\Lambda)}
\le
C
\begin{cases}
\varpi\bigl(q(\mu\flat+1),\flat+1;\mathscr{M}\bigr),
& 1/\mu\in\mathbb N \ \text{or}\ 1/\mu>\flat,\\[1mm]
\varpi\bigl(2q,\flat+1;\mathscr{M}\bigr),
& 1/\mu\notin\mathbb N \ \text{and}\ 1/\mu<\flat,
\end{cases}
\end{equation}
at the collocation points,
\begin{equation}\label{eq:cor51-v-collR}
\max_{x\in\cup_{j=1}^{\mathscr{M}}X_j}
|\psi(x)-\psi_h(x)|
\le
C
\begin{cases}
\varpi\bigl(q(\mu\flat+1),\flat+1;\mathscr{M}\bigr),
& 1/\mu\in\mathbb N \ \text{or}\ 1/\mu>\flat,\\[1mm]
\varpi\bigl(2q,\flat+1;\mathscr{M}\bigr),
& 1/\mu\notin\mathbb N \ \text{and}\ 1/\mu<\flat,
\end{cases}
\end{equation}
and at the regular endpoint \(x=0\),
\begin{equation}\label{eq:cor51-v-terminalR}
|\psi(0)-\psi_h(0)|
\le
C
\begin{cases}
\varpi\bigl(q(\mu\flat+1),\flat;\mathscr{M}\bigr),
& 1/\mu\in\mathbb N \ \text{or}\ 1/\mu>\flat,\\[1mm]
\varpi\bigl(2q,\flat;\mathscr{M}\bigr),
& 1/\mu\notin\mathbb N \ \text{and}\ 1/\mu<\flat.
\end{cases}
\end{equation}
\end{corollary}

\begin{proof}
\noindent{\rm (i)}
By Theorem~\ref{thm:1.1R}, with \(\ell=\flat\),
\[
\varphi(x)
=
\sum_{(i,k)\in\mathcal I_{\nu,\flat}^{R}}
\alpha_{i,k}^{R}(\nu)(b-x)^{i+k(1+\nu)}
+
Y_{\flat+1}^{R}(x,\nu),
\qquad
Y_{\flat+1}^{R}(\cdot,\nu)\in C^{\flat+1}(\Lambda),
\]
where
\[
\mathcal I_{\nu,\flat}^{R}
=
\{(i,k)\in\mathbb N_0^2:\ i+k(1+\nu)<\flat+1\}.
\]
Differentiation gives
\[
\psi(x)
=
-\sum_{\substack{(i,k)\in\mathcal I_{\nu,\flat}^{R}\\
i+k(1+\nu)>0}}
\alpha_{i,k}^{R}(\nu)
\bigl(i+k(1+\nu)\bigr)
(b-x)^{i+k(1+\nu)-1}
+
\frac{d}{dx}Y_{\flat+1}^{R}(x,\nu).
\]
Since \(Y_{\flat+1}^{R}(\cdot,\nu)\in C^{\flat+1}(\Lambda)\), its derivative belongs to
\(C^\flat(\Lambda)\).

Now put \(x=\tau(\rho)\). Since
\[
b-\tau(\rho)=(b-\rho)^{1/\mu},
\]
we obtain
\[
\widehat{\psi}(\rho)
=
-\sum_{\substack{(i,k)\in\mathcal I_{\nu,\flat}^{R}\\
i+k(1+\nu)>0}}
\alpha_{i,k}^{R}(\nu)
\bigl(i+k(1+\nu)\bigr)
(b-\rho)^{\frac{i+k(1+\nu)-1}{\mu}}
+
\widehat Y_\flat^{R}(\rho,\nu),
\]
where
\[
\widehat Y_\flat^{R}(\rho,\nu)
:=
\corr{\bigl(Y_{\flat+1}^{R}\bigr)'(\tau(\rho),\nu)}.
\]
Since \(\nu/\mu\in\mathbb N\),
\[
\frac{i+k(1+\nu)-1}{\mu}
=
\frac{i+k-1}{\mu}
+
k\frac{\nu}{\mu}.
\]
Thus, the possible noninteger exponents in \(\widehat{\psi}\) are determined by
\(1/\mu\). \corr{The same conclusion applies to the transformed remainder
\(\widehat Y_\flat^{R}\), since the Taylor expansion of
\(\bigl(Y_{\flat+1}^{R}\bigr)'\) about \(x=b\) generates powers
\((b-\rho)^{r/\mu}\).}

If
\[
1/\mu\in\mathbb N
\qquad\text{or}\qquad
1/\mu>\flat,
\]
then no noninteger power of order less than or equal to \(\flat\) occurs in
\(\widehat{\psi}\). Hence
\[
\widehat{\psi}\in C^{\flat,1-\gamma}[b-b^\mu,b)
\]
for some \(\gamma>\flat\). Applying Theorem~\ref{thm:5.1R} gives
\[
\|\psi-\psi_h\|_{L^\infty(\Lambda)}
\le
C\mathscr{M}^{-\min\{q\mu\flat,\flat\}},
\qquad
\|\varphi-\varphi_h\|_{L^\infty(\Lambda)}
\le
C\varpi\bigl(q(\mu\flat+1),\flat;\mathscr{M}\bigr).
\]
This proves the first cases in \eqref{eq:cor51-vR} and
\eqref{eq:cor51-uR}.

If
\[
1/\mu\notin\mathbb N
\qquad\text{and}\qquad
1/\mu<\flat,
\]
then \corr{the smallest possible noninteger terminal exponent in
\(\widehat{\psi}\) is \(1/\mu\)}. Hence
\[
\widehat{\psi}\in C^{\flat,1-1/\mu}[b-b^\mu,b).
\]
Theorem~\ref{thm:5.1R}, with \(\gamma=1/\mu\), gives
\[
\|\psi-\psi_h\|_{L^\infty(\Lambda)}
\le
C\mathscr{M}^{-\min\{q,\flat\}},
\qquad
\|\varphi-\varphi_h\|_{L^\infty(\Lambda)}
\le
C\varpi(2q,\flat;\mathscr{M}).
\]
This proves the second cases in \eqref{eq:cor51-vR} and
\eqref{eq:cor51-uR}.

\medskip
\noindent{\rm (ii)}
The proof follows the same argument, using Theorem~\ref{thm:1.1R} with
\(\ell=\flat+1\). Then
\[
\varphi(x)
=
\sum_{(i,k)\in\mathcal I_{\nu,\flat+1}^{R}}
\alpha_{i,k}^{R}(\nu)(b-x)^{i+k(1+\nu)}
+
Y_{\flat+2}^{R}(x,\nu),
\qquad
Y_{\flat+2}^{R}(\cdot,\nu)\in C^{\flat+2}(\Lambda),
\]
where
\[
\mathcal I_{\nu,\flat+1}^{R}
=
\{(i,k)\in\mathbb N_0^2:\ i+k(1+\nu)<\flat+2\}.
\]
After differentiating and using \(x=\tau(\rho)\), the regularity of
\(\widehat{\psi}(\rho)=\psi(\tau(\rho))\) is determined by the powers
\[
(b-\rho)^{\frac{i+k(1+\nu)-1}{\mu}}
=
(b-\rho)^{\frac{i+k-1}{\mu}+k\frac{\nu}{\mu}}.
\]
Since \(\nu/\mu\in\mathbb N\), the possible noninteger powers are governed
by \(1/\mu\), \corr{including those generated by the transformed smooth
remainder}.

If
\[
1/\mu\in\mathbb N
\qquad\text{or}\qquad
1/\mu>\flat+1,
\]
then no noninteger power of order less than or equal to \(\flat+1\) occurs in
\(\widehat{\psi}\). Hence
\[
\widehat{\psi}\in C^{\flat+1,1-\gamma}[b-b^\mu,b)
\]
for some \(\gamma>\flat+1\). Theorem~\ref{thm:5.2R} gives
\[
\|\varphi-\varphi_h\|_{L^\infty(\Lambda)}
+
\max_{x\in\cup_{j=1}^{\mathscr{M}}X_j}
|\psi(x)-\psi_h(x)|
\le
C\varpi\bigl(q(\mu\flat+1),\flat+1;\mathscr{M}\bigr),
\]
and
\[
|\psi(0)-\psi_h(0)|
\le
C\varpi\bigl(q(\mu\flat+1),\flat;\mathscr{M}\bigr).
\]
The same estimates follow when
\[
1/\mu\notin\mathbb N
\qquad\text{and}\qquad
\flat<1/\mu<\flat+1,
\]
because then
\(\widehat{\psi}\in C^{\flat+1,1-1/\mu}[b-b^\mu,b)\), and
Theorem~\ref{thm:5.2R} with \(\gamma=1/\mu\) still gives
\(\min\{\flat,\gamma\}=\flat\). Thus, the first cases in
\eqref{eq:cor51-u-superR}--\eqref{eq:cor51-v-terminalR} hold whenever
\[
1/\mu\in\mathbb N
\qquad\text{or}\qquad
1/\mu>\flat.
\]

If
\[
1/\mu\notin\mathbb N
\qquad\text{and}\qquad
1/\mu<\flat,
\]
then
\(\widehat{\psi}\in C^{\flat+1,1-1/\mu}[b-b^\mu,b)\). Applying
Theorem~\ref{thm:5.2R} with \(\gamma=1/\mu\) gives
\[
\|\varphi-\varphi_h\|_{L^\infty(\Lambda)}
+
\max_{x\in\cup_{j=1}^{\mathscr{M}}X_j}
|\psi(x)-\psi_h(x)|
\le
C\varpi(2q,\flat+1;\mathscr{M}),
\]
and
\[
|\psi(0)-\psi_h(0)|
\le
C\varpi(2q,\flat;\mathscr{M}).
\]
This proves the second cases in
\eqref{eq:cor51-u-superR}--\eqref{eq:cor51-v-terminalR}. The proof is
complete.
\end{proof}

\begin{remark}\label{rem:5.1R}
\corr{Under the assumptions of Corollary~\ref{cor:5.1R}{\rm (ii)} and for
a sufficiently large grading parameter \(q\),} the fractional backward
collocation method can attain
\[
\|\psi-\psi_h\|_{L^\infty(\Lambda)}
\le C\mathscr{M}^{-\flat},
\qquad
\|\varphi-\varphi_h\|_{L^\infty(\Lambda)}
\le C\mathscr{M}^{-\flat-1}.
\]
\corr{These are the maximal convergence orders predicted by the corresponding
local interpolation and superconvergence estimates:}
\(\psi_h\in S_\flat^\mu(\mathscr{P}_{\mathscr{M}})\), while
\(\varphi_h\) is recovered from \(\psi_h\) by the backward integration formula
\eqref{eq:3.1bR}.
\end{remark}

The parameter \(\mu\) plays an important role in the construction of the
method and in the attainable error bounds. Corollary~\ref{cor:5.1R} assumes
\[
\frac{\nu}{\mu}\in\mathbb N,
\]
that is,
\[
\mu=\frac{\nu}{k}
\qquad\text{for some }k\in\mathbb N.
\]
For each admissible value of \(k\), the same corollary indicates how the
grading parameter \(q\) should be chosen in order to obtain the
\corr{highest convergence orders allowed by the estimates}. However, very large
values of \(q\) lead to a strong concentration of mesh points near the terminal
endpoint \(b\). In computations, this may increase the effect of roundoff
errors.

For example, if \(1/\nu\in\mathbb N\), then
\[
\frac1\mu=\frac{k}{\nu}\in\mathbb N.
\]
Corollary~\ref{cor:5.1R}{\rm (i)} then shows that, in order to obtain the
highest order for \(\psi_h\), it is sufficient to take
\[
q\ge\frac1\mu=\frac{k}{\nu},
\]
whereas the corresponding condition for the integrated approximation
\(\varphi_h\) is
\[
q\ge\frac{\flat}{\mu\flat+1}
=
\corr{\frac{k\flat}{k+\nu\flat}}.
\]
Both lower bounds increase with \(k\). Hence, the smallest admissible grading is
obtained when \(k=1\). Thus, the natural practical choice is
\[
\mu=\nu.
\]
\corr{The remaining cases of Corollary~\ref{cor:5.1R} lead to the same practical
recommendation: taking \(k=1\) avoids unnecessarily decreasing \(\mu\) and
does not require stronger mesh grading than larger admissible values of
\(k\).} Therefore, in the numerical experiments below, we take
\(\mu=\nu\).

\begin{remark}\label{rem:5.2R}
For the adjoint Volterra integro-differential equation \eqref{eq:1.1R}, the
weak singularity is located at the terminal endpoint \(b\). If a standard
piecewise polynomial collocation method is used on a uniform mesh, then the
convergence orders for the approximations of \(\varphi\) and
\(\psi=\varphi'\) are limited by this weak terminal regularity.
\corr{For solutions containing the leading fractional term
\((b-x)^{1+\nu}\), the corresponding generic orders cannot exceed}
\(1+\nu\) for \(\varphi_h\) and \(\nu\) for \(\psi_h\). By contrast,
Corollary~\ref{cor:5.1R} shows that the fractional backward collocation method,
combined with a backward graded mesh, can attain the \corr{maximal orders
allowed by the local approximation space}, provided that \(\mu\) and \(q\)
are chosen appropriately. Even on a uniform mesh, \(q=1\), the fractional
approximation space gives higher rates than standard polynomial collocation.
This behavior is illustrated in the numerical experiments in the next section.
\end{remark}

\section{Numerical experiments}
\label{sec:numericalR}

In this section, we present numerical experiments to verify the convergence and
superconvergence estimates established in Corollary~\ref{cor:5.1R}. The
computations are carried out on the backward graded mesh introduced in
Section~\ref{sec2}. For a mesh with \(\mathscr{M}\) subintervals, we define
the global errors by
\[
\mathcal E_{\mathscr{M}}^\varphi
=
\|\varphi-\varphi_h\|_{L^\infty(\Lambda)},
\qquad
\mathcal E_{\mathscr{M}}^\psi
=
\|\psi-\psi_h\|_{L^\infty(\Lambda)},
\]
and the collocation-point and regular-endpoint errors by
\[
\mathcal E_{\mathscr{M},c}^\psi
=
\max_{1\le j\le \mathscr{M}}
\max_{1\le l\le \flat}
\left|\psi(x_{j,l})-\psi_h(x_{j,l})\right|,
\qquad
\mathcal E_{\mathscr{M},0}^\psi
=
\left|\psi(0)-\psi_h(0)\right|.
\]
\corr{For the reported global errors, the maximum norm is approximated on a
grid obtained by subdividing each mesh interval into \(40\) equal parts.}

The corresponding experimental convergence orders are computed from
\[
\mathcal R_{\mathscr{M}}^\varphi
=
\log_2\left(
\frac{\corr{\mathcal E_{\mathscr{M}/2}^\varphi}}
{\mathcal E_{\mathscr{M}}^\varphi}
\right),
\qquad
\mathcal R_{\mathscr{M}}^\psi
=
\log_2\left(
\frac{\corr{\mathcal E_{\mathscr{M}/2}^\psi}}
{\mathcal E_{\mathscr{M}}^\psi}
\right),
\]
and
\[
\mathcal R_{\mathscr{M},c}^\psi
=
\log_2\left(
\frac{\corr{\mathcal E_{\mathscr{M}/2,c}^\psi}}
{\mathcal E_{\mathscr{M},c}^\psi}
\right),
\qquad
\mathcal R_{\mathscr{M},0}^\psi
=
\log_2\left(
\frac{\corr{\mathcal E_{\mathscr{M}/2,0}^\psi}}
{\mathcal E_{\mathscr{M},0}^\psi}
\right).
\]
In the tables below, ``EOC'' denotes the expected order of convergence
predicted by Corollary~\ref{cor:5.1R}.

\begin{example}\label{ex:6.1R}
Consider the adjoint weakly singular Volterra integro-differential equation
\begin{equation}\label{eq:ex1R-model}
\varphi'(x)
=
f(x)+\varphi(x)
+
\int_x^1
(s-x)^{\nu-1}e^{s}\varphi(s)\,ds,
\qquad
0\le x\le1,
\qquad
\varphi(1)=0,
\end{equation}
where $0<\nu<1$. The source term is chosen as
\begin{equation}\label{eq:ex1R-forcing}
\begin{aligned}
f(x)
={}&
\left(
-(1+\nu)(1-x)^\nu
-2(1-x)
-2(1-x)^{1+\nu}
-2(1-x)^2
\right)e^{-x}
\\
&\quad
-
B(\nu,2+\nu)(1-x)^{1+2\nu}
-
B(\nu,3)(1-x)^{2+\nu},
\end{aligned}
\end{equation}
\corr{where $B(\cdot,\cdot)$ denotes the beta function,} so that the exact
solution is
\begin{equation}\label{eq:ex1R-exact}
\varphi(x)
=
\left((1-x)^{1+\nu}+(1-x)^2\right)e^{-x}.
\end{equation}
The fractional-power term $(1-x)^{1+\nu}$ produces the expected loss of
regularity at the terminal endpoint $x=1$, in agreement with
Theorem~\ref{thm:1.1R}.

For $\flat=2$, the collocation parameters are chosen as
$\xi_1=1/4$ and $\xi_2=5/6$, whereas for $\flat=3$, we take
$\xi_1=1/3$, $\xi_2=1/2$, and $\xi_3=2/3$.
These choices satisfy condition~\eqref{eq:4.1R}.
\end{example}

\noindent\textbf{Test 1.}
We first take $\nu=\mu=1/2$. Since $1/\mu=2\in\mathbb N$,
Corollary~\ref{cor:5.1R} predicts, for a sufficiently graded mesh,
$\mathcal R_{\mathscr{M}}^\psi=\flat$,
$\mathcal R_{\mathscr{M}}^\varphi
=\mathcal R_{\mathscr{M},c}^\psi=\flat+1$, and
$\mathcal R_{\mathscr{M},0}^\psi=\flat$.
The results reported in Tables~\ref{tab:right-ex1-global-mu-half} and
\ref{tab:right-ex1-super-mu-half} confirm these estimates. For
$\flat=2$, the global approximation of $\psi$ converges with order two,
while the errors of $\varphi_h$ and the collocation-point approximation of
$\psi$ converge with order three. For $\flat=3$, the corresponding
orders are three and four. The approximation of $\psi(0)$ also converges
with the expected order $\flat$.

We next consider the uniform mesh $q=1$. In this case,
Corollary~\ref{cor:5.1R} gives
$\mathcal R_{\mathscr{M}}^\psi=\mu\flat$,
$\mathcal R_{\mathscr{M}}^\varphi
=\mathcal R_{\mathscr{M},c}^\psi=\mu\flat+1$, and
$\mathcal R_{\mathscr{M},0}^\psi
=\min\{\mu\flat+1,\flat\}$.
The results in Tables~\ref{tab:right-ex1-global-uniform-mu-half} and
\ref{tab:right-ex1-super-uniform-mu-half} are consistent with these
estimates. In particular, the reduced rate of the global derivative error
illustrates the influence of the terminal singularity when no mesh grading is
employed.

\medskip

\noindent\textbf{Test 2.}
We next consider the irrational fractional order
$\nu=\mu=\sqrt{3}/3$. Since $1/\mu=\sqrt{3}<2$, the second case of
Corollary~\ref{cor:5.1R} applies for both $\flat=2$ and $\flat=3$.
With an appropriate grading parameter $q$, the predicted optimal orders are
$\mathcal R_{\mathscr{M}}^\psi=\flat$,
$\mathcal R_{\mathscr{M}}^\varphi
=\mathcal R_{\mathscr{M},c}^\psi=\flat+1$, and
$\mathcal R_{\mathscr{M},0}^\psi=\flat$.
Tables~\ref{tab:right-ex1-global-mu-sqrt3} and
\ref{tab:right-ex1-super-mu-sqrt3} confirm these rates. The global
approximation of $\psi$ approaches orders two and three for
$\flat=2$ and $\flat=3$, respectively, while the approximation of
$\varphi$ and the collocation-point approximation of $\psi$ attain the
superconvergent orders three and four.

Finally, Table~\ref{tab:right-ex1-mixed-mu-sqrt3} illustrates the effect of
suboptimal mesh grading for $\flat=2$. For $q=1$, second-order
convergence is observed for both $\mathcal E_{\mathscr{M}}^\varphi$ and
$\mathcal E_{\mathscr{M},c}^\psi$, in agreement with
Corollary~\ref{cor:5.1R}. For the global derivative error with $q=3/2$,
the theoretical estimate gives
$\mathcal R_{\mathscr{M}}^\psi=\min\{q,\flat\}=3/2$.
The observed rate is approximately two, which is higher than the guaranteed
theoretical rate. This indicates that the corresponding theoretical bound is
conservative for this particular test.

\begin{table}[htbp]
\centering
\caption{Global approximation errors and convergence rates for the adjoint
equation with \(\nu=1/2\) and \(\mu=\nu\).}
\label{tab:right-ex1-global-mu-half}
\renewcommand{\arraystretch}{1.18}
\setlength{\tabcolsep}{5pt}
\begin{tabular}{c|cc|cc|cc|cc}
\toprule
\multirow{3}{*}{\(\mathscr{M}\)}
& \multicolumn{4}{c|}{\(\flat=2\)}
& \multicolumn{4}{c}{\(\flat=3\)}
\\
\cmidrule(lr){2-5}\cmidrule(lr){6-9}
& \multicolumn{2}{c|}{\corr{\(q=3/2\)}}
& \multicolumn{2}{c|}{\corr{\(q=2\)}}
& \multicolumn{2}{c|}{\corr{\(q=8/5\)}}
& \multicolumn{2}{c}{\corr{\(q=2\)}}
\\
\cmidrule(lr){2-3}\cmidrule(lr){4-5}
\cmidrule(lr){6-7}\cmidrule(lr){8-9}
& \(\mathcal E_{\mathscr{M}}^\varphi\)
& \(\mathcal R_{\mathscr{M}}^\varphi\)
& \(\mathcal E_{\mathscr{M}}^\psi\)
& \(\mathcal R_{\mathscr{M}}^\psi\)
& \(\mathcal E_{\mathscr{M}}^\varphi\)
& \(\mathcal R_{\mathscr{M}}^\varphi\)
& \(\mathcal E_{\mathscr{M}}^\psi\)
& \(\mathcal R_{\mathscr{M}}^\psi\)
\\
\midrule
64  & \(2.904\times10^{-6}\) & 2.965
    & \(2.615\times10^{-3}\) & 1.970
    & \(3.142\times10^{-8}\) & 4.004
    & \(3.421\times10^{-5}\) & 2.962 \\
128 & \(3.677\times10^{-7}\) & 2.981
    & \(6.605\times10^{-4}\) & 1.985
    & \(1.958\times10^{-9}\) & 4.004
    & \(4.332\times10^{-6}\) & 2.981 \\
256 & \(4.629\times10^{-8}\) & 2.990
    & \(1.660\times10^{-4}\) & 1.992
    & \(1.221\times10^{-10}\) & 4.004
    & \(5.450\times10^{-7}\) & 2.991 \\
512 & \(5.809\times10^{-9}\) & 2.994
    & \(4.161\times10^{-5}\) & 1.996
    & \(7.613\times10^{-12}\) & 4.003
    & \(6.835\times10^{-8}\) & 2.995 \\
\midrule
EOC & & 3 & & 2 & & 4 & & 3 \\
\bottomrule
\end{tabular}
\end{table}

\begin{table}[htbp]
\centering
\caption{Collocation-point and regular-endpoint errors for \(\psi_h\) in the
adjoint equation with \(\nu=1/2\) and \(\mu=\nu\).}
\label{tab:right-ex1-super-mu-half}
\renewcommand{\arraystretch}{1.18}
\setlength{\tabcolsep}{5pt}
\begin{tabular}{c|cc|cc|cc|cc}
\toprule
\multirow{3}{*}{\(\mathscr{M}\)}
& \multicolumn{4}{c|}{\(\flat=2\)}
& \multicolumn{4}{c}{\(\flat=3\)}
\\
\cmidrule(lr){2-5}\cmidrule(lr){6-9}
& \multicolumn{2}{c|}{\corr{\(q=3/2\)}}
& \multicolumn{2}{c|}{\corr{\(q=1\)}}
& \multicolumn{2}{c|}{\corr{\(q=8/5\)}}
& \multicolumn{2}{c}{\corr{\(q=6/5\)}}
\\
\cmidrule(lr){2-3}\cmidrule(lr){4-5}
\cmidrule(lr){6-7}\cmidrule(lr){8-9}
& \(\mathcal E_{\mathscr{M},c}^\psi\)
& \(\mathcal R_{\mathscr{M},c}^\psi\)
& \(\mathcal E_{\mathscr{M},0}^\psi\)
& \(\mathcal R_{\mathscr{M},0}^\psi\)
& \(\mathcal E_{\mathscr{M},c}^\psi\)
& \(\mathcal R_{\mathscr{M},c}^\psi\)
& \(\mathcal E_{\mathscr{M},0}^\psi\)
& \(\mathcal R_{\mathscr{M},0}^\psi\)
\\
\midrule
64  & \(3.300\times10^{-6}\) & 3.009
    & \(6.643\times10^{-4}\) & 1.994
    & \(2.869\times10^{-8}\) & 3.976
    & \(7.506\times10^{-6}\) & 2.988 \\
128 & \(4.089\times10^{-7}\) & 3.013
    & \(1.664\times10^{-4}\) & 1.997
    & \(1.801\times10^{-9}\) & 3.993
    & \(9.421\times10^{-7}\) & 2.994 \\
256 & \(5.068\times10^{-8}\) & 3.012
    & \(4.164\times10^{-5}\) & 1.999
    & \(1.125\times10^{-10}\) & 4.001
    & \(1.180\times10^{-7}\) & 2.997 \\
512 & \(6.288\times10^{-9}\) & 3.011
    & \(1.042\times10^{-5}\) & 1.999
    & \(7.014\times10^{-12}\) & 4.004
    & \(1.477\times10^{-8}\) & 2.998 \\
\midrule
EOC & & 3 & & 2 & & 4 & & 3 \\
\bottomrule
\end{tabular}
\end{table}

\begin{table}[htbp]
\centering
\caption{Global approximation errors and convergence rates for the adjoint
equation with \(\nu=1/2\), \(\mu=\nu\), and \(q=1\).}
\label{tab:right-ex1-global-uniform-mu-half}
\renewcommand{\arraystretch}{1.18}
\setlength{\tabcolsep}{5pt}
\begin{tabular}{c|cc|cc|cc|cc}
\toprule
\multirow{3}{*}{\(\mathscr{M}\)}
& \multicolumn{4}{c|}{\(\flat=2\)}
& \multicolumn{4}{c}{\(\flat=3\)}
\\
\cmidrule(lr){2-5}\cmidrule(lr){6-9}
& \multicolumn{4}{c|}{\corr{\(q=1\)}}
& \multicolumn{4}{c}{\corr{\(q=1\)}}
\\
\cmidrule(lr){2-5}\cmidrule(lr){6-9}
& \(\mathcal E_{\mathscr{M}}^\varphi\)
& \(\mathcal R_{\mathscr{M}}^\varphi\)
& \(\mathcal E_{\mathscr{M}}^\psi\)
& \(\mathcal R_{\mathscr{M}}^\psi\)
& \(\mathcal E_{\mathscr{M}}^\varphi\)
& \(\mathcal R_{\mathscr{M}}^\varphi\)
& \(\mathcal E_{\mathscr{M}}^\psi\)
& \(\mathcal R_{\mathscr{M}}^\psi\)
\\
\midrule
128 & \(1.540\times10^{-6}\) & 2.095
    & \(1.395\times10^{-3}\) & 1.138
    & \(1.314\times10^{-7}\) & 2.638
    & \(8.372\times10^{-5}\) & 1.646 \\
256 & \(3.676\times10^{-7}\) & 2.067
    & \(6.529\times10^{-4}\) & 1.095
    & \(2.167\times10^{-8}\) & 2.600
    & \(2.786\times10^{-5}\) & 1.587 \\
512 & \(8.891\times10^{-8}\) & 2.048
    & \(3.118\times10^{-4}\) & 1.066
    & \(3.647\times10^{-9}\) & 2.571
    & \(9.541\times10^{-6}\) & 1.546 \\
\midrule
EOC & & 2 & & 1 & & 2.5 & & 1.5 \\
\bottomrule
\end{tabular}
\end{table}

\begin{table}[htbp]
\centering
\caption{Collocation-point and regular-endpoint errors for \(\psi_h\) in the
adjoint equation with \(\nu=1/2\), \(\mu=\nu\), and \(q=1\).}
\label{tab:right-ex1-super-uniform-mu-half}
\renewcommand{\arraystretch}{1.18}
\setlength{\tabcolsep}{5pt}
\begin{tabular}{c|cc|cc|cc|cc}
\toprule
\multirow{3}{*}{\(\mathscr{M}\)}
& \multicolumn{4}{c|}{\(\flat=2\)}
& \multicolumn{4}{c}{\(\flat=3\)}
\\
\cmidrule(lr){2-5}\cmidrule(lr){6-9}
& \multicolumn{4}{c|}{\corr{\(q=1\)}}
& \multicolumn{4}{c}{\corr{\(q=1\)}}
\\
\cmidrule(lr){2-5}\cmidrule(lr){6-9}
& \(\mathcal E_{\mathscr{M},c}^\psi\)
& \(\mathcal R_{\mathscr{M},c}^\psi\)
& \(\mathcal E_{\mathscr{M},0}^\psi\)
& \(\mathcal R_{\mathscr{M},0}^\psi\)
& \(\mathcal E_{\mathscr{M},c}^\psi\)
& \(\mathcal R_{\mathscr{M},c}^\psi\)
& \(\mathcal E_{\mathscr{M},0}^\psi\)
& \(\mathcal R_{\mathscr{M},0}^\psi\)
\\
\midrule
128 & \(2.077\times10^{-6}\) & 2.198
    & \(1.664\times10^{-4}\) & 1.997
    & \(2.994\times10^{-7}\) & 2.640
    & \(5.207\times10^{-7}\) & 3.009 \\
256 & \(4.710\times10^{-7}\) & 2.141
    & \(4.164\times10^{-5}\) & 1.999
    & \(4.925\times10^{-8}\) & 2.604
    & \(6.426\times10^{-8}\) & 3.019 \\
512 & \(1.099\times10^{-7}\) & 2.100
    & \(1.042\times10^{-5}\) & 1.999
    & \(8.257\times10^{-9}\) & 2.576
    & \(7.869\times10^{-9}\) & 3.030 \\
\midrule
EOC & & 2 & & \corr{2} & & 2.5 & & \corr{2.5} \\
\bottomrule
\end{tabular}
\end{table}

\begin{table}[htbp]
\centering
\caption{Global approximation errors and convergence rates for the adjoint
equation with \(\nu=\sqrt{3}/3\) and \(\mu=\nu\).}
\label{tab:right-ex1-global-mu-sqrt3}
\renewcommand{\arraystretch}{1.18}
\setlength{\tabcolsep}{5pt}
\begin{tabular}{c|cc|cc|cc|cc}
\toprule
\multirow{3}{*}{\(\mathscr{M}\)}
& \multicolumn{4}{c|}{\(\flat=2\)}
& \multicolumn{4}{c}{\(\flat=3\)}
\\
\cmidrule(lr){2-5}\cmidrule(lr){6-9}
& \multicolumn{2}{c|}{\corr{\(q=3/2\)}}
& \multicolumn{2}{c|}{\corr{\(q=2\)}}
& \multicolumn{2}{c|}{\corr{\(q=2\)}}
& \multicolumn{2}{c}{\corr{\(q=3\)}}
\\
\cmidrule(lr){2-3}\cmidrule(lr){4-5}
\cmidrule(lr){6-7}\cmidrule(lr){8-9}
& \(\mathcal E_{\mathscr{M}}^\varphi\)
& \(\mathcal R_{\mathscr{M}}^\varphi\)
& \(\mathcal E_{\mathscr{M}}^\psi\)
& \(\mathcal R_{\mathscr{M}}^\psi\)
& \(\mathcal E_{\mathscr{M}}^\varphi\)
& \(\mathcal R_{\mathscr{M}}^\varphi\)
& \(\mathcal E_{\mathscr{M}}^\psi\)
& \(\mathcal R_{\mathscr{M}}^\psi\)
\\
\midrule
64  & \(2.944\times10^{-6}\) & 2.965
    & \(2.604\times10^{-3}\) & 1.968
    & \corr{\(5.966\times10^{-8}\)} & \corr{4.000}
    & \(1.063\times10^{-4}\) & 2.925 \\
128 & \(3.726\times10^{-7}\) & 2.982
    & \(6.584\times10^{-4}\) & 1.984
    & \corr{\(3.727\times10^{-9}\)} & \corr{4.001}
    & \(1.364\times10^{-5}\) & 2.962 \\
256 & \(4.688\times10^{-8}\) & 2.991
    & \(1.655\times10^{-4}\) & 1.992
    & \corr{\(2.329\times10^{-10}\)} & \corr{4.001}
    & \(1.727\times10^{-6}\) & 2.981 \\
512 & \(5.881\times10^{-9}\) & 2.995
    & \(4.149\times10^{-5}\) & 1.996
    & \corr{\(1.455\times10^{-11}\)} & \corr{4.000}
    & \(2.173\times10^{-7}\) & 2.991 \\
\midrule
EOC & & 3 & & 2 & & 4 & & 3 \\
\bottomrule
\end{tabular}
\end{table}

\begin{table}[htbp]
\centering
\caption{Collocation-point and regular-endpoint errors for \(\psi_h\) in the
adjoint equation with \(\nu=\sqrt{3}/3\) and \(\mu=\nu\).}
\label{tab:right-ex1-super-mu-sqrt3}
\renewcommand{\arraystretch}{1.18}
\setlength{\tabcolsep}{5pt}
\begin{tabular}{c|cc|cc|cc|cc}
\toprule
\multirow{3}{*}{\(\mathscr{M}\)}
& \multicolumn{4}{c|}{\(\flat=2\)}
& \multicolumn{4}{c}{\(\flat=3\)}
\\
\cmidrule(lr){2-5}\cmidrule(lr){6-9}
& \multicolumn{2}{c|}{\corr{\(q=3/2\)}}
& \multicolumn{2}{c|}{\corr{\(q=1\)}}
& \multicolumn{2}{c|}{\corr{\(q=2\)}}
& \multicolumn{2}{c}{\corr{\(q=3/2\)}}
\\
\cmidrule(lr){2-3}\cmidrule(lr){4-5}
\cmidrule(lr){6-7}\cmidrule(lr){8-9}
& \(\mathcal E_{\mathscr{M},c}^\psi\)
& \(\mathcal R_{\mathscr{M},c}^\psi\)
& \(\mathcal E_{\mathscr{M},0}^\psi\)
& \(\mathcal R_{\mathscr{M},0}^\psi\)
& \(\mathcal E_{\mathscr{M},c}^\psi\)
& \(\mathcal R_{\mathscr{M},c}^\psi\)
& \(\mathcal E_{\mathscr{M},0}^\psi\)
& \(\mathcal R_{\mathscr{M},0}^\psi\)
\\
\midrule
64  & \(3.277\times10^{-6}\) & 2.994
    & \(6.625\times10^{-4}\) & 1.994
    & \corr{\(8.572\times10^{-8}\)} & 3.950
    & \(1.379\times10^{-5}\) & 2.978 \\
128 & \(4.093\times10^{-7}\) & 3.001
    & \(1.660\times10^{-4}\) & 1.997
    & \corr{\(5.442\times10^{-9}\)} & 3.977
    & \(1.737\times10^{-6}\) & 2.989 \\
256 & \(5.103\times10^{-8}\) & 3.004
    & \(4.154\times10^{-5}\) & 1.998
    & \corr{\(3.424\times10^{-10}\)} & \corr{3.990}
    & \(2.179\times10^{-7}\) & 2.994 \\
512 & \corr{\(6.361\times10^{-9}\)} & 3.004
    & \(1.039\times10^{-5}\) & 1.999
    & \corr{\(2.146\times10^{-11}\)} & \corr{3.996}
    & \(2.729\times10^{-8}\) & 2.997 \\
\midrule
EOC & & 3 & & 2 & & 4 & & 3 \\
\bottomrule
\end{tabular}
\end{table}

\begin{table}[htbp]
\centering
\caption{Selected errors and convergence rates for the adjoint equation with
\(\nu=\sqrt{3}/3\), \(\mu=\nu\), and \(\flat=2\).}
\label{tab:right-ex1-mixed-mu-sqrt3}
\renewcommand{\arraystretch}{1.18}
\setlength{\tabcolsep}{7pt}
\begin{tabular}{c|cc|cc|cc}
\toprule
\multirow{3}{*}{\(\mathscr{M}\)}
& \multicolumn{2}{c|}{\(\mathcal E_{\mathscr{M}}^\varphi\)}
& \multicolumn{2}{c|}{\(\mathcal E_{\mathscr{M}}^\psi\)}
& \multicolumn{2}{c}{\(\mathcal E_{\mathscr{M},c}^\psi\)}
\\
\cmidrule(lr){2-3}\cmidrule(lr){4-5}\cmidrule(lr){6-7}
& \multicolumn{2}{c|}{\corr{\(q=1\)}}
& \multicolumn{2}{c|}{\corr{\(q=3/2\)}}
& \multicolumn{2}{c}{\corr{\(q=1\)}}
\\
\cmidrule(lr){2-3}\cmidrule(lr){4-5}\cmidrule(lr){6-7}
& Error
& \(\mathcal R_{\mathscr{M}}^\varphi\)
& Error
& \(\mathcal R_{\mathscr{M}}^\psi\)
& Error
& \(\mathcal R_{\mathscr{M},c}^\psi\)
\\
\midrule
128 & \(1.230\times10^{-6}\) & 2.105
    & \(3.721\times10^{-4}\) & 1.991
    & \(1.398\times10^{-6}\) & 2.187 \\
256 & \(2.928\times10^{-7}\) & 2.070
    & \(9.332\times10^{-5}\) & 1.995
    & \(3.198\times10^{-7}\) & 2.128 \\
512 & \(7.086\times10^{-8}\) & 2.047
    & \(2.337\times10^{-5}\) & 1.998
    & \(7.526\times10^{-8}\) & 2.087 \\
\midrule
EOC & & 2 & & \(\frac32\) & & 2 \\
\bottomrule
\end{tabular}
\end{table}

\section{Conclusion}\label{sec:conclusion}

We developed \corr{an optimal-order fractional backward collocation method}
for adjoint Volterra integro-differential equations with weakly singular
kernels. The terminal condition and the backward Volterra structure
\corr{induce} a weak singularity at the terminal endpoint, which
\corr{may lead to a deterioration of the convergence order} for standard
polynomial collocation methods.

We first established a regularity result showing that the solution is
continuously differentiable, while its second derivative may exhibit a weak
singularity near the terminal endpoint. This regularity structure motivates
the use of a graded mesh refined toward \(b\) together with a fractional
polynomial approximation space adapted to the terminal singularity.

The convergence analysis \corr{provides global error estimates for both
\(\psi=\varphi'\) and \(\varphi\), together with superconvergence estimates
for \(\varphi\), for \(\psi\) at the collocation points, and for \(\psi\) at
the regular endpoint \(x=0\)} under the prescribed condition on the
collocation parameters. In particular, suitable choices of the fractional
parameter \(\mu\) and the grading exponent \(q\) allow the method to attain
\corr{the maximal convergence orders predicted by the corresponding
approximation and superconvergence estimates}.

The numerical experiments \corr{are consistent with} the theoretical
convergence and superconvergence rates and demonstrate the effectiveness of
the fractional backward collocation method for adjoint weakly singular
Volterra integro-differential equations with terminal-endpoint singularities.
\corr{The analysis also provides a framework for extending the method to
more general terminal-value problems and Volterra-type equations exhibiting
weak endpoint singularities.}

\section*{Data availability statement}
No datasets were generated or analyzed during the current study.

\section*{Declarations}

\section*{Conflict of interest}
The authors declare that they have no conflict of interest.


\bibliographystyle{elsart-num-sort}
	\bibliography{Bibfileamc}

\end{document}